\documentclass[12pt]{amsart}
\usepackage{enumerate,mathtools,overpic}
\usepackage{enumitem}
\usepackage{amsmath, amssymb, amsthm,latexsym}
\usepackage{graphicx}

\usepackage{hyperref}

\makeatletter
\@namedef{subjclassname@2010}{%
  \textup{2010} Mathematics Subject Classification}
\makeatother

\setenumerate{itemsep=3pt,topsep=3pt}
\setenumerate[1]{label=\upshape(\roman*)}

\newcommand{\mybar}[1]{\makebox[0pt]{$\phantom{#1}\overline{\phantom{#1}}$}#1}
 
\DeclareRobustCommand{\SkipTocEntry}[5]{}

\title{Quasiconformal and geodesic trees}%
\author{Mario Bonk}
\address{Department of Mathematics, University of California, 
Los Angeles, CA 90095, USA}
\email{mbonk@math.ucla.edu}

\author{Daniel Meyer}
\address{Department of Mathematical Sciences, University of Liverpool,
  Liverpool L69 7ZL,  United Kingdom}
\email{dmeyermail@gmail.com}

\date{\today}

\newcommand\C{{\mathbb C}}

\newcommand\N{{\mathbb N}}

\newcommand\R{{\mathbb R}}

\newcommand\ba{\backslash}

\newcommand\dist{\operatorname{dist}}
\newcommand\diam{\operatorname{diam}}
\newcommand\id{\operatorname{id}}
\newcommand\inte{\operatorname{int}}

\renewcommand\:{\colon}
\newcommand\sub {\subseteq}
\newcommand\ra {\rightarrow}

\renewcommand\sub{\subset}

\def\length{\operatorname{length}}

\newcommand{\CDach}{\widehat{\C}}
\newcommand{\D}{\mathbb{D}}

\newcommand\ga{\gamma}

\newcommand\eps{\epsilon}

\providecommand{\abs}[1]{\lvert#1\rvert}

\newcommand{\X} {\mathbf{X}}

\newcommand{\V} {\mathbf{V}}

\newcommand{\qT} {\mathbf{T}}

\newtheorem{theorem}{Theorem}[section]

\newtheorem{proposition}[theorem]{Proposition}
\newtheorem{cor}[theorem]{Corollary}

\newtheorem{lemma}[theorem]{Lemma}

\newcommand{\dia}{\operatorname{\mathsf{dd}}}

\theoremstyle{definition}

\newtheorem{definition}[theorem]{Definition}

\theoremstyle{remark}

\numberwithin{equation}{section}

\begin{document}





\abstract{A quasiconformal tree is a  metric tree that is doubling and of bounded turning. We prove that every quasiconformal tree is quasisymmetrically equivalent  to a geodesic tree with Hausdorff dimension arbitrarily close to 1.}
\endabstract

\maketitle

\tableofcontents

\section{Introduction}\label{s:Intro}

An important question in geometric analysis is whether a given
metric space (belonging to some class of spaces) is
geometrically
equivalent to
a model space in a natural way. Many results in
mathematics can be seen from this perspective (such as the existence
of isothermal or conformal co\-ordinates on surfaces or the
Riemann mapping theorem).  For general metric spaces there are
various ways to interpret geometric equivalence: up to isometric
or up to bi-Lipschitz equivalence, for example.  In the present
paper the relevant notion of geometric equivalence is based on a
class of homeomorphisms that are close to conformal or
quasiconformal maps in a classical complex-analytic context,
namely quasisymmetries.

 By definition, a homeomorphism $f\: X\to Y$ between metric spaces $(X,d_X)$ and $(Y, d_Y)$ is said to be \emph{quasisymmetric} or a {\em quasisymmetry},  if there exists a homeomorphism $\eta\:[0,\infty)\to[0,\infty)$ (playing the role of a control function for distortion) such that
$$
\frac{d_Y(f(x),f(y))}{d_Y(f(x),f(z))}\le\eta\left(\frac{d_X(x,y)}{d_X(x,z)}\right)
$$
for all distinct points $x,y,z\in X$. 
 The 
 composition of two quasisymmetries (when defined) and the inverse 
 of a quasisymmetry are quasisymmetric. 
So if we call two metric spaces $X$ and $Y$
{\em quasisymmetrically equivalent} if there exists a quasisymmetry
$f\:X\ra Y$, then we have  a notion of geometric equivalence for metric spaces. Since every bi-Lipschitz homeomorphism is a quasisymmetry, this is a weaker, and hence more flexible, notion than bi-Lipschitz (or even isometric) equivalence (for more background and related discussions see \cite[Section 4.1]{BM17} and 
\cite[Chapters 10--12]{He}). 

The {\em quasisymmetric uniformization problem} (see
\cite{Bo06}) asks for natural conditions when a given metric space $X$ from a class of spaces is quasisymmetrically equivalent to some model space $Y$. This problem is relevant in various contexts. For example, the  Kapovich-Kleiner conjecture in geometric group theory (see \cite[Conjecture 6]{KK00}) amounts to the problem of showing that every Sierpi\'nski carpet arising as the boundary of a Gromov hyperbolic group is quasisymmetrically equivalent to a ``round" Sierpi\'nski carpet
(see \cite{Bo11} for a related discussion).  

The prototypical instance of a quasisymmetric uniformization
 result is the characterization by Tukia and V\"ais\"al\"a of metric spaces quasisymmetrically
equivalent to the unit interval $[0,1]$.
 In order to formulate  their  theorem   we need two definitions.  

We say that a metric space $(X,d)$ is of {\em bounded turning} if
there exists a constant $K\ge 1$ such that for all $x,y\in X$
there exists a compact connected set $E\sub X$ with $x,y\in E$
and 
$$ \diam (E) \le K d(x,y).$$ 
 In this case, we say that $(X,d)$ is of $K$-bounded turning. 

A metric space $(X,d)$ is {\em doubling} if there exists a
constant $N\in \N$ (the {\em doubling constant} of $X$) such that
each ball in $X$ of radius $R>0$ can be covered by $N$ (or fewer)
balls of radius $R/2$.

 Tukia and V\"ais\"al\"a showed that 
 a metric space $J$ homeomorphic to $[0,1]$ is quasisymmetrically equivalent to $[0,1]$ if and only if
it is doubling and of bounded turning (see \cite{TV}). In other
words, one can ``straighten out''  the arc $J$ (which may well
have Hausdorff dimension $>1$) to the interval $[0,1]$ by a
quasisymmetry.


In the present paper, we study the quasisymmetric uniformization
problem for metric trees. By definition, a {\em (metric) tree} is a compact, connected, and
locally connected metric space $(\qT, d)$ that contains at least
two distinct points and has the following property: if
$x,y\in \qT$, then there exists a unique arc in $\qT$ with
endpoints $x$ and $y$. This arc is denoted by $[x,y]$. We allow
$x=y$ here, in which case we consider $[x,y]=\{x\}$ as a {\em
  degenerate} arc. 
  
  The underlying topological space of a tree is  often called a {\em dendrite} in the literature. Since we are mostly interested in metric properties and want to emphasize this metric aspect, we  prefer the name tree for these objects.  Motivated by the  Tukia-V\"ais\"al\"a result and the connection with quasiconformal geometry, we introduce the following terminology.

\begin{definition}\label{def:qt}
A metric
tree is {\em quasiconformal} if it is doubling and of bounded turning.
\end{definition} 
In the following, we usually call a quasiconformal tree a {\em
  qc-tree} for brevity.

Trees appear in many contexts in mathematics, for example as
Julia sets of polynomials. The Julia set
$\mathcal{J}(P)$ of the polynomial $P(z)=z^2+i$ is a tree
(see \cite[Example after Theorem~V.4.2]{CG}).
Actually,
$\mathcal{J}(P)\sub \C$ is a qc-tree if  it is equipped with the ambient Euclidean metric on $\C$.   Indeed,  $\mathcal{J}(P)$
is of bounded turning as easily follows from the fact that
$\C\setminus \mathcal{J}(P)$ is a John domain
(see \cite[Theorem~VII.3.1]{CG}).
Since every subset of a Euclidean space (such as the
complex plane $\C$) is doubling, $\mathcal{J}(P)$ is doubling.

In  analogy  to  the Tukia-V\"ais\"al\"a theorem  
one can raise the question whether all arcs in  a qc-tree can be straightened out 
simultaneously by a quasisymmetry.  For a precise formulation of this question the following concept is relevant. 

A metric space $(X,d)$ is called {\em geodesic} if
any two points $x,y\in X$  can be joined by a {\em geodesic
segment}, i.e., by an arc $[x,y]$ with endpoints $x$ and $y$ whose
length is equal to $d(x,y)$.  


The following statement is the main result of this paper.

\begin{theorem}
  \label{main} 
  Every  quasiconformal tree is
  quasisymmetrically equivalent to a geodesic tree.
\end{theorem}

Every  arc that is doubling and of bounded turning is
a qc-tree. This implies that  Theorem~\ref{main} includes the Tukia-V\"ais\"al\"a theorem as a special case, and so can it  be viewed as a generalization.  

Various  improvements and variants  of  Theorem~\ref{main} are conceivable. For
example, one can ask whether additional assumptions yield quasisymmetric
equivalence to a single specified space. We consider a question of this type in the follow-up paper \cite{BM}, where it is shown that a qc-tree
is quasisymmetrically
equivalent to the {\em continuum self-similar tree} (as defined in \cite{BT18}) if and only if it is {\em trivalent}  and {\em uniformly branching} (see \cite{BM} for the relevant definitions).


Another natural question is ``how small'' we can make the 
geodesic tree $T$ that is the quasisymmetric
image of the given  qc-tree $\qT$. If $\dim_H T$ denotes the Hausdorff dimension of $T$, then clearly $\dim_H T\ge 1$, because 
$T$ always contains a non-degenerate arc. We will show that  $\dim_H T$ can actually be arbitrarily close to $1$ and will establish the following improved version of Theorem~\ref{main}.  
 
\begin{theorem}
  \label{thm:qtree_confdim1}
 If $\qT$ is a quasiconconformal tree and $\alpha>1$, then $\qT$ is quasisymmetrically equivalent to a geodesic tree $T$ with $\dim_HT\le\alpha$. 
\end{theorem}

In general, one cannot achieve  $\dim_HT=1$ here.  An example when this is not possible can be found in \cite{BiT01} (see also \cite[Theorem 1.6] {Az15} for a  general related statement).
If $\qT$ is the continuum self-similar tree and $T$ is any tree
that is quasisymmetrically equivalent to $\qT$, then actually
$\dim_HT>1$.

The 
 \emph{conformal dimension} $\operatorname{confdim}(X)$ of a
metric space $X$  is
defined as the infimum of all Hausdorff dimensions of metric
spaces $Y$ that are quasisymmetrically equivalent to $X$. We
refer to \cite{MT10} for more background on this concept. Theorem~\ref{thm:qtree_confdim1} implies the following immediate 
consequence. 

\begin{cor} If $\qT$ is a quasiconformal tree, then 
$\operatorname{confdim}(\qT)=1$. 
\end{cor} 
This last statement is not new, but was originally proved by Kinneberg \cite[Proposition 2.4]{Kin}.

We will now summarize the main ingredients for the proofs of Theorems~\ref{main} and~\ref{thm:qtree_confdim1}. The basic idea 
 is to define a new geodesic metric $\varrho$ on the given 
 qc-tree $(\qT,d)$ so that the 
identity map $\id_{\qT}\: (\qT,d) \ra (\qT, \varrho)$ is a quasisymmetry. In order to define $\varrho$, we will carefully
choose a sequence of decompositions $\X^n$  of $\qT$ into subtrees. We call the elements $X^n$  in $\X^n$ tiles of level $n$ or $n$-tiles.  To each $n$-tile $X^n$ we will assign a weight 
$w(X^n)$  by an inductive process on the level $n\in \N$. These
weights can then be used to define a distance function
$\varrho_n$ on $\qT$: one  infimizes the total length with
respect to this weight over chains of $n$-tiles from one point in $\qT$ to another (see  \eqref{eq:deflengthP} and \eqref{eq:defrhon}). We will show  that 
with our choices, the limit 
\begin{equation} \label{eq:linrnexists} \varrho(x,y)=\lim_{n\to \infty}\varrho_n(x,y)
\end{equation}
exists for all 
$x,y\in \qT$ (Lemma~\ref{lem:rhonlimexis}) and defines a geodesic metric on $\qT$ (Lemma~\ref{lem:rho_metric}).  
We have  $\diam_\varrho(X)\asymp w(X)$ for the $\varrho$-diameter
of each  tile $X$ (see Proposition~\ref{prop:mett}~\ref{item:mett1}). So in a sense the metric $\varrho$ is a ``conformal" deformation of the original metric $d$ on $\qT$ controlled by the weight $w(X)$ near each tile $X$.  The fact that  
$\id_{\qT}\: (\qT,d) \ra (\qT, \varrho)$ is a quasisymmetry
can then easily be derived from geometric properties of
tiles (see Lemma~\ref{lem:id_qs}). 
Theorem~\ref{main} follows.

The choice of the weights and hence the construction of $\varrho$ involves a parameter $\eps_0>0$. We will see that 
if we choose $\eps_0$ close to $0$, then the  Hausdorff dimension of $(\qT, \varrho)$ is close to $1$. This immediately gives 
Theorem~\ref{thm:qtree_confdim1}. 

The main difficulty in this general approach is how to define the decompositions $\X^n$. It is a natural idea to ``cut" the tree $\qT$  into subtrees by using auxiliary points. We will indeed follow this procedure by defining an ascending  sequence of finite sets $\V^1\sub \V^2 \sub \dots$ that we use to cut  $\qT$. More precisely, the tiles of level $n$ are  precisely  the closures of the complementary components of $\V^n$, i.e., the closures of the components of 
$\qT\setminus \V^n$. The construction of the sets $\V^n$ involves
a (small) parameter
$\delta\in (0,1)$.
For each $n$-tile $X^n$ we will then have $\diam_d(X^n)\asymp\delta^n$. All of this looks natural and even straightforward, but there is a surprising  subtlety here. 
Namely, one might expect that the $n$-vertices, i.e., the elements in $\V^n$ used for cutting the tree,  should be \emph{branch points} of $\qT$
(points $b\in \qT$ such that $\qT\setminus \{b\}$ has at least three components); indeed, at least  
on an intuitive level, cutting $\qT$ in a branch point should
result in branches  with reduced topological or metric
complexity.
This was exactly the procedure in the recent paper 
\cite{BT18}, where topological characterizations of metric trees
were given.  We also use this idea in our forthcoming paper
\cite{BM}.
However,
in the present context, 
cutting our given qc-tree $\qT$ at a branch point $b$ leads to the  problem that  we cannot expect 
good uniform control for  the size  of the components of $\qT\setminus\{b\}$, because some of these components might be very small.  

For this reason, we cut our given qc-tree  $\qT$  at \emph{double
  points} $v\in \qT$, i.e., points $v$ such that  $\qT\setminus \{v\}$ has precisely two components. These double points $v$ are chosen so that the two components of $\qT\setminus \{v\}$ are not too small and so that 
 $v$ stays away from the branch points of $\qT$ in a precise
quantitative way (see \eqref{eq:double_Delta} and
\eqref{eq:d_x_branch}; the relevant definitions can be found in \eqref{eq:DT} and \eqref{eq:HT}).

The paper is organized as follows. In Section~\ref{sec:diameter-distance} we review some basic topological facts about trees.
We also show that in a tree $\qT$ of bounded turning one can replace 
the original metric up to bi-Lipschitz equivalence by a {\em
  diameter metric} $d$. It is characterized by the property  that
$\diam\, [x,y]=d(x,y)$ for all $x,y\in \qT$. The change to a
diameter metric will allow us to make some simplifications of our arguments. In Section~\ref{sec:place_sun} we will prove a general fact of independent interest: 
if  on an arc some points cast a ``shadow" satisfying suitable conditions, then one can always find a ``place in the sun". We use this to  find double points in a qc-tree $\qT$ with quantitative separation from  branch points (see  Proposition~\ref{prop:branchaway}).  

In Section~\ref{sec:goodpts} we introduce the somewhat technical
concept of a $(\beta, \ga)$-good double point at scale $\Delta>0$.
We show that with suitable choices of the parameters cutting the
qc-tree $\qT$ in a maximal $\Delta$-separated set of
$(\beta, \ga)$-good double points at scale $\Delta>0$ results in
pieces that have diameter comparable to $\Delta$
(Proposition~\ref{prop:maxdouble}). This fact is used in
Section~\ref{sec:subdividing-tree} to define the subdivisions of
$\qT$ into tiles as discussed above. 
We  record  various
statements about the geometric properties of these tile
decompositions. Weights of tiles are then defined in
Section~\ref{sec:weights-vertices}. There we establish the facts
about weights that are needed later on. In
Section~\ref{sec:constr-geod-metr} we define the metric $\varrho$
and show that it is geodesic. The proof of
Theorem~\ref{main} is then completed in
Section~\ref{sec:quasisymmetry} and the proof of
Theorem~\ref{thm:qtree_confdim1} is given in
Section~\ref{sec:lower-hausd-dimens}. We conclude with remarks and  open problems   in
Section~\ref{sec:open-probl-concl}.

\subsection{Notation}
\label{sec:notation}

We summarize some  notation used throughout  this paper. 

When an object $A$ is defined to be another object $B$, we write $A\coloneqq B$
for emphasis.  Two non-negative quantities $a$ and $b$ are said to be {comparable}
if there is a constant $C\ge 1$ (usually depending on some  ambient para\-meters)   such that
\begin{equation*}
  \frac{1}{C}a\leq b\leq C a.
\end{equation*}
We then write 
$a\asymp b$.
The constant $C$ is referred to  as $C(\asymp)$. Similarly, we write
  $a\lesssim b$ or $b\gtrsim a$,  
if there is a constant $C>0$ such that $a\leq C b$, and  refer to
the constant $C$ as $C(\lesssim)$ or $C(\gtrsim)$. If we want to emphasize the parameters $\alpha$, $\beta, \dots$ on which $C$ depends, then we write $C=C(\alpha, \beta, \dots)$. 


We use the standard notation $\N=\{1,2,\dots\}$ and
$\N_0=\{0,1, 2, \dots \}$.

The cardinality of a set $X$ is denoted by $\#X$ and the identity
map on $X$ by $\id_X$.  Let $(X,d)$ be a metric space, $a\in X$,
and $r>0$. We denote by $B_d(a,r)=\{x\in X:d(a,x)< r\}$ the open ball
and by $\mybar{B}_d(a,r)=\{x\in X: d(a,x)\le r\}$ the closed
ball of radius $r$ centered at $a$. If $A,B\sub X$, we let
$\diam_d(A)$ be the diameter, $\overline A$ be the closure of $A$
in $X$, $\inte(A)$ be the interior  of $A$
in $X$, and
\begin{equation*}
  \dist_d(A,B)\coloneqq\inf\{d(x,y): x\in A,\, y\in B\}
\end{equation*}
be the distance of $A$ and $B$. If $x\in X$, we set
$\dist_d(x, A)\coloneqq\dist_d(\{x\}, A)$. We drop the subscript
$d$ from our notation for $B_d(a,r)$, etc., if the metric $d$ is
clear from the context.

\section{Auxiliary facts}
\label{sec:diameter-distance}

In this section we collect some auxiliary statements that  will be used 
later. 

Let $(X,d)$ be a metric space. A set $S\subset X$ is called
\emph{$s$-separated} for some $s>0$ if all 
distinct points $x,y\in S$ satisfy $d(x,y)\geq s$. Such a
set $S$ is a {\em maximal} $s$-separated set if $S$ is not
contained  in a strictly larger subset of $X$ that is also
$s$-separated. Every $s$-separated set $S\sub X$ is contained in
a maximal $s$-separated set $S'\sub X$.  If  $X$ is compact, then  every $s$-separated set $S\sub X$ must be  finite. 

If the space $(X,d)$ is doubling (as defined in the
introduction), then for each 
$0<\lambda<1$ there is a number $N'=N'(\lambda, N)\in \N$ only
depending on $\lambda$ and the doubling constant $N$ of $X$  such
that the following condition is true:  if  $s>0$ and $S\sub X$ is
a $\lambda s$-separated set contained in a ball $B(x,s)$ with
$x\in X$, then $S$ contains at most $N'$ points. Conversely, if
this condition is true for some $0<\lambda<1$ and $N'\in \N$,
then $X$ is doubling with a doubling constant $N=N(\lambda, N')$
only depending on $\lambda$ and $N'$ (see
\cite[Exercise~10.17]{He}).

The doubling property is preserved under quasisymmetries, and in particular 
under bi-Lipschitz maps;
in general though, the doubling constant will change (see \cite[Theorem~10.18]{He}).

 An {\em arc} $J\sub X$  is a set homeomorphic 
 to  the  unit interval $[0,1]\sub \R$. A {\em (metric) arc} $(J,d)$ is a metric space homeomorphic  to  $[0,1]$.  The points $a,b\in J$ corresponding to $0,1\in [0,1]$ are called the {\em endpoints} 
 of $J$.  We denote by $\partial J\coloneqq\{a,b\}$ the set of endpoints of $J$, and by $\inte(J)\coloneqq J\setminus \partial J$ the set of {\em interior points} of $J$.  
 
 We require an elementary lemma. 

\begin{lemma}
  \label{lem:divide_A_N}
  Let $(J,d)$ be an  arc and $n\geq 2$ be an integer. Then we
  can decompose  $J$ into $n$ non-overlapping   subarcs of equal diameter 
  $\Delta\ge \tfrac 1n \diam(J)$.  
  \end{lemma}
  More explicitly, decomposing  $J$ into $n$ non-overlapping
  subarcs means that  we can find arcs $I_1, \dots, I_n\sub J$ with pairwise disjoint interiors such  
  $J=I_1\cup \dots \cup I_n $.  

\begin{proof} The existence of a decomposition of $J$ into $n$
  non-overlapping subarcs of equal diameter is proved in 
\cite[Lemma~2.2]{Me11} (see also \cite[Lemma~2]{Kul} for a related statement in greater generality). If we denote this diameter by $\Delta>0$, then we must have $\diam(J)\le n\Delta$ as follows from the triangle inequality. 
\end{proof}

 We now summarize some simple facts about  trees. There is a rich literature on the underlying topological spaces, usually called dendrites. We refer to \cite[Chapter V]{Wh}, \cite[Section \S 51 VI]{Ku68},  \cite[Chapter~X]{Na}, and the references in these sources for more on the 
subject. 

By definition, a metrizable topological space $X$ is called a {\em dendrite} if $X$  is a {\em Peano continuum} (i.e., it is compact, connected, and locally connected), and $X$  does not contain any Jordan curve (i.e., a homeo\-morphic image of the unit circle). A dendrite is called {\em non-degenerate} if it contains more than one point. The following statement reconciles our notion of a metric tree  with the notion of a  dendrite. 

\begin{proposition}
  \label{prop:tr=den}
  Let $\qT$ be a metric space. Then $\qT$ is a tree if and only
  if $\qT$ is a non-degenerate dendrite.
\end{proposition}

\begin{proof}
  ``$\Rightarrow$''
  If $\qT$ is a tree, then it is a Peano continuum and contains more than one point. Moreover, $\qT$ cannot contain a Jordan curve $J$. Indeed, if $\qT$ contains 
the Jordan curve $J$, then
any two distinct points $x,y\in J$ can  be connected by at least 
two distinct arcs in $\qT$, namely the two subarcs of $J$ with endpoints 
$x$ and $y$. This is impossible,  because $\qT$ is a tree. It follows that $\qT$ is a non-degenerate dendrite.

\smallskip
``$\Leftarrow$''
Conversely, suppose $\qT$ is a non-degenerate dendrite. Since $\qT$ is a Peano continuum, it  is arc-connected, i.e., for any two distinct points  $x,y\in \qT$ there exists an arc $\alpha \sub \qT$ with endpoints $x$ and $y$ (see \cite [Theorem 8.23]{Na}). This arc is unique, because if there exists  an arc $\beta\sub \qT$ with $\beta\ne \alpha$ and endpoints $x$ and $y$, then it is easy to see that $\alpha \cup \beta\sub \qT$ contains a Jordan curve. 
This is impossible, because $\qT$ is a dendrite. It follows that $\qT$ is indeed a tree. 
\end{proof} 

Let $\qT$ be a tree. Then for all points $x,y\in \qT$ with $x\ne y$, there 
 exists a unique arc in $\qT$  joining $x$ and $y$, i.e., it has the endpoints 
 $x$ and $y$. We use the notation $[x,y]$  for this unique arc. 
 It is convenient to allow $x=y$ here. Then $[x,y]$ denotes a {\em degenerate} arc consisting  only of the point $x=y$. 
 Sometimes we want to remove one or both endpoints from the arc $[x,y]$. Accordingly, we define
  \begin{equation*}
  (x,y] \coloneqq [x,y]\setminus\{x\}, 
  \quad 
  [x,y) \coloneqq [x,y]\setminus\{y\},
  \quad
  (x,y) \coloneqq [x,y]\setminus\{x,y\}.  
\end{equation*}  
If $\ga$ is the image of any path in $\qT$ joining $x$ and $y$,
then necessarily $[x,y]\sub \ga$.
  
A subset $X$ of a tree $(\qT,d)$ is called a {\em subtree} of
$\qT$ if $X$ equipped with the restriction of the metric $d$ is
also a tree.  One can show that $X\sub \qT$ is a subtree of $\qT$
if and only if $X$ contains at least two points and is closed and
connected. See \cite[Lemma~3.3]{BT18} for a simple direct
argument; to justify this, one can also invoke
Proposition~\ref{prop:tr=den} and the fact that a closed and
connected subset of a dendrite is a dendrite (see \cite[Corollary
10.6] {Na}). If $X$ is a subtree of $\qT$, then $[x,y]\sub X$ for
all $x,y\in X$.

\begin{lemma}
  \label{lem:top_T}
  Let $(\qT,d)$ be a tree and  $V\subset \qT$ be a finite set. Then
  the following statements are true:

  \begin{enumerate}
    \item 
    \label{item:top_T1}
    Two points $x,y\in \qT\setminus V$ lie in the same
    component of $\qT\setminus V$ if and only if $[x,y]\cap
    V=\emptyset$.  
    \item 
    \label{item:top_T2} If $U$ is a component of $\qT\setminus V$, then $U$ is an open set and 
    $\overline{U}$ is a subtree of $\qT$ with  $\partial \overline U\sub \partial U\sub V$. 
  \item
    \label{item:top_T3}
    If $U$ and $W$ and  are distinct components of $\qT\setminus
    V$, then $\overline{U}$ and $\overline{W}$ 
    have at most one point in common. Such a common point belongs to $V$, and is a boundary point of both  $\overline{U}$ and $\overline{W}$. 
  \end{enumerate}
\end{lemma}

\begin{proof}
  \ref{item:top_T1}
  Since $V\sub \qT$ is a finite set, it is closed in $\qT$. So
  $\qT\setminus V$ is an open subset of $\qT$. Since $\qT$ is
  locally connected, each component $U$ of $\qT\setminus V$ is
  open.  Moreover, as an open and connected subset of the Peano
  continuum $\qT$, such a component $U$ is arc-connected (see
  \cite[Theorem 8.26] {Na}). So if two points
  $x,y\in \qT\setminus V$ lie in the same component $U$ of
  $\qT\setminus V$, then there exists an arc $\gamma$ in $U$
  joining $x$ and $y$. Then $\gamma=[x,y] \sub U$, and so
  $[x,y]\cap V=\emptyset$.

 Conversely, if $x,y\in \qT\setminus V$ and
 $[x,y]\cap V=\emptyset$, then $[x,y]$ is a connected subset of
 $\qT\setminus V$. Hence there exists a component $U$ of
 $\qT\setminus V$ with $ [x,y]\sub U$; so $x$ and $y$ lie in the
 same component $U$ of $\qT\setminus V$.
    
 \smallskip
 \ref{item:top_T2}
 If $U$ is a component of $\qT\setminus V$, then $U$ is an open
 set (as we have seen in the proof of \ref{item:top_T1}) and
 $\overline U$ is a subtree of $\qT$ (as follows from the
 characterization of subtrees discussed before the lemma).
  
 The inclusion $\partial \overline U\sub \partial U$ is true for all sets $U\sub \qT$. It remains to show $\partial U\sub V$. Indeed, if $x\in \partial U$, then  $x$ cannot belong to $U$ (since $U$ is open) or any other component $W$ of $\qT\setminus V$ (because otherwise $U\cap  W\ne \emptyset$); so $x$ lies in the complement of $\qT\setminus V$ in $\qT$, i.e., $x\in V$.

\smallskip
 \ref{item:top_T3} Suppose $U$ and $W$ are distinct
  components of $\qT\setminus V$. 
  Since $U$ and $W$ are disjoint open subsets of $\qT$ by \ref{item:top_T2}, no interior point of  $\overline U$ can belong to
  $\overline W$, and no interior point of $\overline W$ can belong to $\overline U$. 
   Hence $$\overline U\cap \overline W=\partial \overline U\cap \overline W=\partial \overline U\cap \partial \overline W\sub 
   \partial U\cap \partial W\sub V$$ by \ref{item:top_T2}. In particular, $\overline U\cap \overline W$ is a subset of the finite set  $V$,  and any point in $\overline U\cap \overline W$  must be a boundary point of  both  $\overline U$ and  $\overline W$.
   
 Actually,  $\overline U\cap \overline W$ consists of at most one point; otherwise, $\overline U\cap \overline W$  contains two distinct 
  points $x$ and $y$, and hence the infinite set $[x,y]$, because $\overline U$ and 
  $\overline W$ are subtrees of $\qT$. This is impossible, because  the set $\overline U\cap \overline W\sub V$ is finite. 
  \end{proof}


Let  $\qT$ be a tree, $p\in \qT$, and $U$ be  a  component of
 $\qT\setminus\{p\}$. Then $U\ne \qT$ is open, and so
 $\partial U\ne \emptyset$, because $\qT$ is connected. So by Lemma~\ref{lem:top_T}~\ref{item:top_T2} we have 
 $ \emptyset\ne \partial U \sub \{p\}$. Hence   $\partial U=\{p\}$ and so 
 $\overline U=U\cup \{p\}$. Then  $B\coloneqq \overline U=U \cup\{p\}$ is a subtree of $\qT$, called 
 a {\em branch} of  $p$ (in $\qT$).

The components $U$ of any open subset $W$ of a tree 
$\qT$ form a {\em null sequence} in the following sense: 
for each $\eps>0$ there are only finitely many such components
$U$ with $\diam(U)\ge \eps$.
In particular, the number of components of $W$
is finite or countably infinite. This follows from a more general fact about open subsets of {\em hereditarily locally connected} metric continua; see \cite[p.~90, Corollary (2.2)]{Wh} or 
\cite[p.~269, Theorem 3] {Ku68}. Note that we can apply this result
by Proposition~\ref{prop:tr=den} and  because every dendrite  is hereditarily locally connected (this is explicitly stated in \cite[Corollary 10.5]{Na} and follows from the fact,
 mentioned above,  that every subcontinuum of a dendrite is a dendrite). 
 
In particular, each point  $p$ in a tree $\qT$ can have at most countably many distinct 
complementary components $U$ and hence there are only countably
many  
distinct branches $B$ of  $p$. 
Only finitely many of these branches can have a diameter  exceeding a given positive number (for a direct proof of these facts see also \cite[Section~3]{BT18}). This implies that we can label the branches $B_n$ of $p$ by numbers $n=1,2,3,\dots$ so that 
$$ \diam(B_1)\ge \diam(B_2)\ge \diam(B_3)\ge \dots\, . $$

If there are precisely two such branches, then we call $p$ a {\em
  double point} of $\qT$ and define 
\begin{equation}\label{eq:DT}
D_\qT(p)=\diam(B_2).
\end{equation}
So $D_\qT(p)$ is the diameter of the smallest branch of a double point $p$.  

If there are at least three branches of $p$, then $p$ is called a
{\em branch point} of $\qT$. In this case, we set
\begin{equation}\label{eq:HT}
  H_{\qT}(p)=\diam(B_3). 
\end{equation} 
So $H_\qT(p)$ is the diameter of the third largest branch of  $p$.

The following statement gives a criterion how to detect branch
points.  

\begin{lemma}
  \label{lem:tripcrit}
  Let $(\qT,d)$ be a tree, $b,x_1,x_2,x_3\in \qT$ with
  $b\ne x_1, x_2, x_3$ and suppose that the sets $[x_1,b)$,
  $[x_2,b)$, $[x_3,b)$ are pairwise disjoint.  Then the points
  $x_1,x_2,x_3$ lie in different components of
  $\qT\setminus\{b\}$ and $b$ is a branch point of $\qT$.
\end{lemma}

\begin{proof}
  This is  \cite[Lemma 3.6]{BT18}. For the reader's convenience we reproduce the argument.  
  The arcs $[x_1,b]$ and $[x_2,b]=[b, x_2]$ have only the point $b$ in common. So   their union  $[x_1,b]\cup [b, x_2]$ 
is an arc and this arc must be equal to $[x_1,x_2]$. Hence $b\in [x_1,x_2]$ which by Lemma~\ref{lem:top_T}~\ref{item:top_T1} implies that $x_1$ and $x_2$ lie in different components of $\qT\ba\{b\}$. A similar argument shows that $x_3$ must be contained in a component of 
$\qT\ba\{b\}$ different from the components containing $x_1$ and  $x_2$. In particular, $\qT\ba\{b\}$ has at least three components and so $b$ is a branch point of $\qT$. The statement follows. 
\end{proof} 

The tree $(\qT,d)$  is of $K$-bounded
turning with $K\ge 1$ (as defined in the introduction) if and
only if
\begin{equation*}
  \diam\, [x,y]\le Kd(x,y)
\end{equation*}
for all $x,y\in \qT$. Here and in the following, $\diam\,[x,y]$ instead of $\diam([x,y])$
denotes the diameter of the arc $[x,y]$; we omit the 
parentheses for better readability.  

We define the
\emph{diameter distance} on $\qT$ by 
\begin{equation}
  \label{eq:def_dia}
  \dia(x,y) := \diam\, [x,y]
\end{equation}
for  $x,y\in \qT$. We record some properties of this distance function.

\begin{lemma} Let $(\qT, d)$ be a metric tree. Then the following statements are true: 
  \label{lem:prop_dia}
  \mbox{}
  \begin{enumerate}
  \item
    \label{item:dd_metric}
    $\dia$ is a metric on $\qT$.
  \item 
    \label{item:dd_diam}
    For each arc $J\subset \qT$ we have 
    \begin{equation*}
      \diam_{\dia}(J) = \diam (J), 
    \end{equation*}
    where $\diam_{\dia}$ denotes the diameter with respect to $\dia$. 
  \item 
    \label{item:dd_1bt}
    $(\qT,\dia)$ is of $1$-bounded turning. 
    
      \item 
    \label{item:dd_biL}
    $(\qT,d)$ is of $K$-bounded turning for $K\ge 1$ if and only if the identity map
     $\id_{\qT}\colon
    (\qT,d) \to (\qT, \dia)$ is $K$-bi-Lipschitz.

  \end{enumerate}
\end{lemma}

\begin{proof}
  This is  \cite[Lemma 2.1]{Me11}, but we will include the simple proof for the convenience of the reader.

  \ref{item:dd_metric} All properties of a metric for
  $\dia$ are immediate except the triangle inequality which follows from the fact that if $x,y,z\in \qT$, then $[x,z]\sub [x,y]\cup [y,z]$.


\smallskip 
\ref{item:dd_diam} For all $x,y\in J$,
we have $d(x,y)\le\dia(x,y)$,  and 
so $\diam (J)\le\diam_{\dia}(J)$. Moreover,  for all $x,y\in J$ we have $[x,y]\sub J$. Hence 
$\dia(x,y)\le\diam (J)$; so $\diam_{\dia} (J) \le\diam (J)$ and the statement follows. 

\smallskip
 \ref{item:dd_1bt} This follows directly from \ref{item:dd_diam}, since
$$\dia(x,y)=\diam\, [x,y]=\diam_{\dia}[x,y]$$ for all
$x,y\in \qT$. 

\smallskip
\ref{item:dd_biL}  If $(\qT,d)$ is of $K$-bounded
turning, then for all
$x,y\in \qT$ we have  
$$
  \dia(x,y)=\diam\, [x,y]\le K d(x,y) \le K \dia(x,y).
$$
Thus the identity map $\id_{\qT}\colon (\qT,d) \to (\qT,\dia)$ is
$K$-bi-Lipschitz.  Conversely, if this map is $K$-bi-Lipschitz, 
then for all $x,y\in \qT$,  
$$
  \diam\, [x,y]=\diam_{\dia}[x,y]=\dia(x,y)\le K d(x,y).
$$
Therefore, $(\qT, d)$ is of $K$-bounded turning.
\end{proof}

We say a metric $d$ on a metric tree $\qT$ is a {\em diameter
  metric} if $d(x,y)=\diam\,[x,y]$ for all
$x,y\in \qT$. In this case, $d=\dia$, where $\dia$ is defined as
in \eqref{eq:def_dia}.

Suppose $(\qT,d)$ is a qc-tree, i.e., a tree that is doubling and
of bounded turning.  Then the previous lemma implies that
$(\qT, \dia)$ is bi-Lipschitz equivalent, and in particular
quasisymmetrically equivalent, to $(\qT, d)$. Moreover,
$(\qT, \dia)$ is of $1$-bounded turning and also doubling, since
the latter condition is invariant under bi-Lipschitz equivalence;
so $(\qT, \dia)$ is also a qc-tree. This implies that in order to
prove Theorems~\ref{main} and~\ref{thm:qtree_confdim1}, we are
reduced to the case that the qc-tree in question carries a
diameter metric. This reduction makes the proofs somewhat easier,
but we still face major problems, because there is no obvious way
to turn a diameter metric into a geodesic metric by a
quasisymmetry.

For the rest of the paper we will assume that $(\qT, d)$ is a 
qc-tree that is equipped with a diameter metric $d$. Nothing essential changes if we rescale the metric. So we may also assume that 
$\diam(\qT)=1$.  We will denote the doubling constant of $\qT$ by $N$ throughout the paper.

\section{Sun and shadow} 
\label{sec:place_sun}

In this section we will prove a statement, 
Proposition~\ref{prop:branchaway},  that will allow us to find
double points in our given qc-tree $\qT$ that stay away from the branch
points of $\qT$ in a geometrically controlled manner. In the
formulation of the proposition, we use the function defined in \eqref{eq:HT}.

\begin{proposition}
  \label{prop:branchaway}
  There exists a constant $\ga=\ga(N)>0$ only depending on the
  doubling constant $N$ of $\qT$ with the following property: if
  $\Delta>0$ and $J\sub \qT$ is an arc with $\diam(J)\ge \Delta$,
  then there exists a double point $x\in J$ of $\qT$ such that
  \begin{equation*}
    d(x,b)\ge \ga \min\{H_{\qT}(b), \Delta\}
  \end{equation*}
  for all branch points $b\in \qT$.
\end{proposition}

To prove this statement, we require two  auxiliary facts. 

\begin{lemma}
  \label{lem:arc-decomp}
  Let $(J,d)$ be a metric arc equipped with a diameter metric $d$,
  $J'\sub J$ be an arc, and $A\sub J$ be a set with $\#(A\cap J')\le M$, where $M\in \N$. 
  Then there exists an arc $I\sub J'$ such that
  \begin{equation*}
    \diam(I)
    =
    \tfrac1{6M}\diam(J')
    \text{ and }
    \dist(I, A\cup \partial J')
    \geq
    \tfrac1{6M}\diam(J'). 
  \end{equation*}
\end{lemma}
The statement is somewhat  technical, because three arcs $I\sub
J'\sub J$ are involved, but in this form the lemma  will be
useful for us  later on.

\begin{proof}[Proof of Lemma~\ref{lem:arc-decomp}]
  The construction that follows is illustrated in
  Figure~\ref{fig:JJpI}.
By Lemma~\ref{lem:divide_A_N} we can decompose 
$J'$ into $M+1$  non-overlapping arcs $J'_1, \dots, J'_{M+1}$ of
equal diameter $\Delta$.
\begin{figure}[h]
  \begin{overpic}[width=12.7cm, tics=10, 
    ]{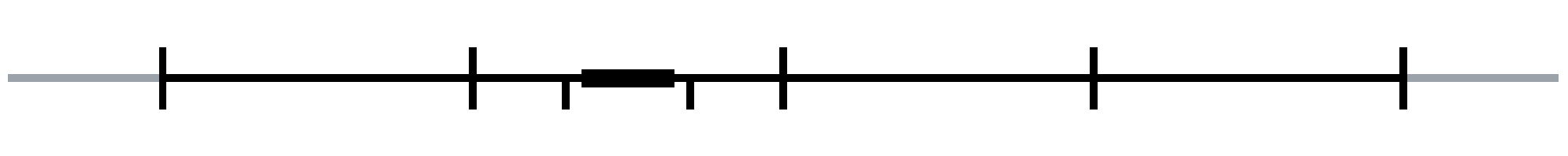}
    \put(18,8){$J'_1$}
    \put(34.5,8){$J''=J'_k$}
    \put(75.5,8){$J'_{M+1}$}
    \put(86,-0.7){$J'$}
    \put(97,-0.7){$J$}
    \put(36.5,-0.5){{$I\!\subset\! I_2$}}
    \put(32,-0.5){{$I_1$}}
    \put(46.6,-0.5){{$I_3$}}
  \end{overpic}
  \caption{The  arcs in the proof of Lemma~\ref{lem:arc-decomp}.}
  \label{fig:JJpI}
\end{figure}
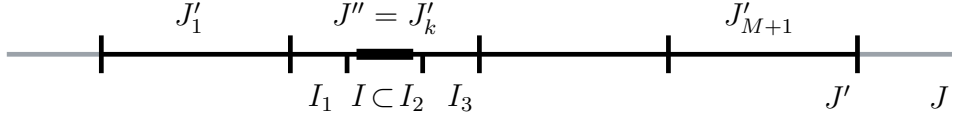

We have 
$$\Delta\ge \tfrac 1{M+1} \diam(J')\ge \tfrac 1{2M} \diam(J').$$
 Since $\#(A\cap J')\le M$ and $J'_1, \dots, J'_{M+1}$  have pairwise disjoint interiors,   by the pigeon-hole principle 
   there exists $k\in \{1, \dots, M+1\}$ such that for $J''\coloneqq J'_k$ we have $\inte(J'')\cap A=\emptyset. $ 
 We subdivide 
 $J''$ into three non-overlapping arcs $I_1, I_2, I_3$ of equal diameter. 
 Then 
 \begin{equation}\label{eq:arcdiambd} \diam(I_i)\ge  \tfrac 13 \diam(J'') \ge  \tfrac 1{6M} \diam(J') \end{equation}
 for $i=1,2,3$.  
 We may assume that $I_1$ contains one endpoint of $J''$,  
 $ I_3$ contains the other endpoint, and $I_2$ is the ``middle" arc in the decomposition of $J''$. It easily follows from \eqref{eq:arcdiambd} and  the intermediate value theorem that there exists an arc $I\sub I_2$ with $ \diam (I)=\tfrac 1{6M} \diam(J')$. 
 
 If $a\in A\cup \partial J'$, then $a\not \in \inte(J'')$. So if we travel from a point 
 $x\in I$ 
 to the point $a$ along $[x,a]\sub J$, we must traverse $I_1$ or $I_3$. Since 
 $d$ is a diameter metric, \eqref{eq:arcdiambd}  implies that 
 \begin{equation*}
   d(x,a)
   =\diam\,[x,a]
   \ge 
   \min\{\diam(I_1),  \diam(I_3)\} \ge 
   \tfrac 1{6M} \diam(J'). 
 \end{equation*}
 Hence $ \dist(I, A\cup \partial J') \ge \tfrac 1{6M}
 \diam(J')$. The statement follows.
\end{proof}

\begin{lemma}[Ein Platz an der Sonne\footnote{Mit f\"{u}nf Mark
    sind Sie dabei!}]
  \label{lem:Sonne}
  Let $(J,d)$ be a metric arc equipped with a diameter metric $d$,
  and $S\: J\ra [0, \diam(J)]$ be a
  function. Suppose that there is a constant $M\in \N$ such that
  for all subarcs $I\subset J$ we have 
  \begin{equation}
    \label{eq:no_large_comp}
    \# \{ p\in I: S(p)\ge \diam(I)\} \le M. 
  \end{equation}
  Then there exists a constant $\sigma=\sigma(M)>0$ and a point
  $x\in J$ such that $d(x,p)\ge \sigma S(p)$ for all $p\in J$.
\end{lemma}
  
In other words, the set
$J\setminus \bigcup_{p\in J}B(p, \sigma S(p))$ is non-empty (here we use  the  convention that $B(p,0) =\emptyset$).  If we think of
each point $p\in J$ with $S(p)>0$ as ``casting a shadow'' of radius  $\sigma S(p)$ around
$p$, then the lemma says that  the union of all shadows
does not cover $J$,  and so there is a ``place in the sun''.

\begin{proof} 
  Without loss of generality we may assume that $\diam(J)=1$.
  Consider the set $A\coloneqq \{p\in J : S(p)>0\}$. 
  Let $\lambda:=1/(6M)$ and define
  $A_n:= \{ p\in A: S(p)\ge \lambda^{n}\}$ for $n\in \N_0$.
  Obviously, $A_n\subset A_{n+1}$ for $n\in \N_0$ and
  $A=\bigcup_{n\in \N_0} A_n$.  We will inductively define arcs
  $J_n\subset J$ for $n\in \N_0$ such that
  $J_0\supset J_1 \supset J_2\dots $,
  $ \diam(J_n)=\lambda^{n}$
for all $n\in \N_0$, and 
  $ \dist(J_n, A_{n-1})\ge \lambda^{n} $ 
  for all $n\in \N$. 
  
  We set $J_0\coloneqq J$. Suppose arcs $J_0, \dots, J_n$ with the desired properties have already been defined for some $n\in \N_0$. Then by our hypotheses $ \#(A_n\cap J_n) \le M. $ It follows from Lemma~\ref{lem:arc-decomp} that we can find an arc 
  $J_{n+1}\sub J_n$ with 
  $$ \diam(J_{n+1}) = \tfrac1{6M} \diam (J_n) =\lambda \diam (J_n)=\lambda^{n+1} $$
  and  $ \dist(J_{n+1}, A_n)\ge \lambda^{n+1}$. 
  Hence $J_{n+1}$ has the desired properties, and we can continue the process indefinitely. 
 
 We have  $\bigcap_{n\in \N_0} J_n\ne \emptyset$, and so we can pick a point 
$x\in J$   that lies in all arcs $J_n$. If $p\in A$ is arbitrary, then there exists a smallest $n\in \N_0$ such that $p\in A_n$. 
Then $S(p)\in [ \lambda^n, \lambda^{n-1})$, and so
$$ d(x,p)\ge \dist(x, A_n)\ge \lambda^{n+1}\ge \lambda^2S(p).$$
So if we  choose $\sigma=\lambda^2=1/(36M^2)$, then $x$ is a
point as desired.  
\end{proof} 

\begin{proof}[Proof of Proposition~\ref{prop:branchaway}] 
Let $\Delta>0$ and suppose $J\sub \qT$ is an arc with 
$\diam(J)\ge \Delta$.  Then $J=[u,v]$, where $u,v\in\qT$   are the endpoints  of $J$. We set $S(p)=\Delta$ for $p\in \{u,v\}$, 
$S(p)=\min\{H_\qT(p),\Delta\}$ for a branch point $p\in (u,v)$,   and $S(p)=0$ for all other points $p\in   (u,v)$.  Since $\diam(J)\ge \Delta$ and $0\le S(p)\le \Delta$ for $p\in J$, we can consider $S$ as a function $S\colon J\ra [0, \diam(J)]$. 

\smallskip 
{\em Claim.} There exists a constant $M=M(N)\in \N$ such that 
  for all arcs $I\subset J$ we have 
  \begin{equation} \label{eq:arccond}
      \# \{ p\in I: S(p)\ge \diam (I)\} \le M. 
  \end{equation}
  
In other words, $S$ satisfies the hypotheses of Lemma~\ref{lem:Sonne} with a constant 
$M=M(N)$ only depending on the doubling constant $N$ of $\qT$. 

\smallskip 
To see this, fix an arc $I\sub J$  and let
$R\coloneqq \{p \in \inte (I): S(p)\ge \rho\}$, where
$\rho \coloneqq \diam(I)>0$. Each point $r\in R$ is a branch point of 
$\qT$ and there exists a large component $U_r$ of
 $\qT\setminus \{r\}$ that is disjoint from $I$, but attached to $I$ through the point  $r$. There are $\#R$ such components. 
 The doubling property then gives 
 a bound on $\#R$ only depending on $N$. In the following we
 present the details of
this argument, illustrated in
Figure~\ref{fig:boundR}. 

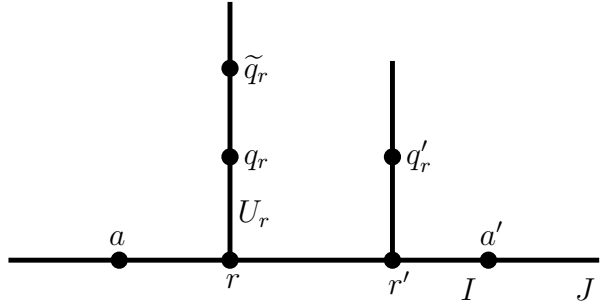
\begin{figure}[h]
  \begin{overpic}[width=8cm, tics=10, 
    ]{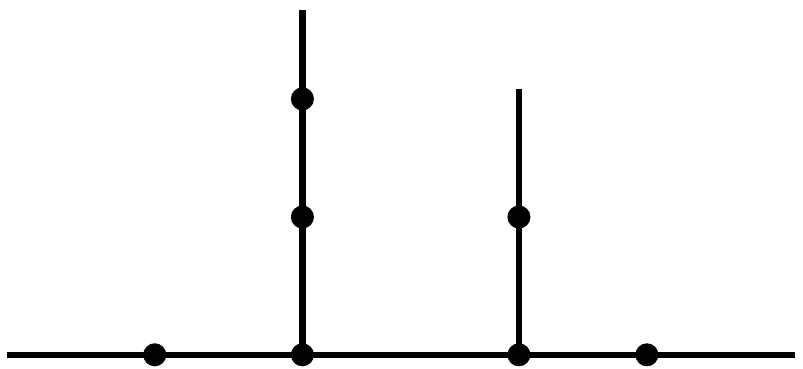}
    \put(95,-5){$J$}
    \put(76,-5){$I$}
    \put(79.2,4.2){$a'$}
    \put(17.6,4.2){$a$}
    \put(37,-3){$r$}
    \put(64,-4){$r'$}
    \put(40,17.7){$q_r$}
    \put(40,32){$\widetilde{q}_r$}
    \put(39,8){$U_r$}
    \put(67,17.7){$q'_r$}
  \end{overpic}
  \caption{Bounding the number of elements in $R$.}
  \label{fig:boundR}
\end{figure}

We have  $I=[a,a']$, where $a,a'\in I$ are the endpoints of $I$. 
Consider an arbitrary  point $r\in R\sub \inte(I)=(a,a')$. 
Then  $r$ is a branch point of $\qT$ with $H_\qT(r) \ge S(r) \ge  \rho>0$.
Each of the connected sets  $[a,r)$ and $(r,a']$ is contained in a component  of $\qT\setminus \{r\}$. Hence there must be another component 
$U_r$ of $\qT\setminus \{r\}$ with  
$\diam(U_r) \ge H_\qT(r) \ge \rho$ that  is disjoint from $I=[a,r)\cup \{r\} \cup (r,a']$.   There exists  a point 
$\widetilde q_r\in U_r$ with $d(\widetilde q_r, r)\ge \diam(U_r)/2\ge \rho/2$; otherwise,  $\mybar{U}_r=U_r\cup \{r\}\sub B(r, \rho/2)$ and so $\diam(U_r) \le \diam(\mybar{U}_r)<\rho$, which is a contradiction.



Then $(r, \widetilde q_r]\sub U_r$, and it easily follows from the intermediate value theorem that we can find a point $q_r\in 
(r, \widetilde q_r]\sub U_r$ with $d(q_r, r)=\rho/2$. We have 
$$ d(a, q_r)\le d(a,r)+d(r, q_r)\le \diam(I)+\rho/2=3\rho/2, $$
and so $q_r\in \overline B(a,3\rho/2)$.

If  $r,r'\in R$  with $r\ne r'$, then the corresponding points $q_r$ and  $q_{r'}$ lie in different components of $\qT\setminus\{r'\}$. 
To see this, note that $\mybar{U}_r=U_r\cup\{r\}$ is a connected set with 
$$  \mybar{U}_r=U_r\cup\{r\}\sub (\qT\setminus I)\cup (\qT\setminus \{r'\})\sub \qT\setminus \{r'\},$$
and so  $\mybar{U}_r$  is contained in a  component of $\qT\setminus \{r'\}$. 
In particular, $q_r\in U_r$ and $r\in   I\setminus \{r'\}$ lie in the same component of $\qT\setminus \{r'\}$. On the other hand, $q_{r'}$ was chosen from a component $U_{r'}$ of $\qT\setminus \{r'\}$ that does  not contain any point of $I$.

 Since $q_r$ and  $q_{r'}$ lie in different components of $\qT\setminus\{r'\}$, 
 Lemma~\ref{lem:top_T}~\ref{item:top_T1} implies that $r'\in [q_r, q_{r'}]$.   In particular, 
\begin{equation*}
  d(q_r, q_{r'})
  =
  \diam\, [q_r, q_{r'}]
  \ge 
  d(r', q_{r'})=\rho/2. 
\end{equation*}

So the points $q_r$, $r\in R$, have pairwise mutual distance 
$\ge \rho/2$ and are all contained in the ball $\overline B(a, 3\rho/2)$. It follows that $\#R$ is bounded by a constant only depending on $N$ (see the discussion in the beginning of Section~\ref{sec:diameter-distance}). Since the endpoints of $I$ are not contained in $R$, we have to possibly increase this bound by $2$  to obtain a bound as in \eqref{eq:arccond} with a constant $M=M(N)$. The Claim follows.




\smallskip
Lemma~\ref{lem:Sonne} now guarantees the existence of a point 
$x\in J$ such that 
$$ d(x,p)\ge \sigma S(p) $$ for all $p\in J$, where 
$\sigma=\sigma(M)=\sigma(N)>0$ can be chosen to depend only on $M$ and hence on $N$. We may assume that $0<\sigma\le 1$, and so $\sigma^2\le \sigma$. 

We claim that the statement of the proposition is true with $\ga\coloneqq \sigma^2/2$ which only depends on $N$. To see this,  let $b\in \qT$ be an arbitrary branch point of $\qT$. As we travel from $b$ to $x$ along the arc $[b,x]$, there is a first point $r\in J$. 
We now consider two cases depending on the location of $r$.

\smallskip
{\em Case 1.} $r\in B(u, \sigma \Delta/ 2) \cup 
B(v, \sigma \Delta/ 2)$.  In this case,  we may assume $r\in B(u, \sigma \Delta/ 2)$. Since $S(u)=\Delta$, by  choice of $x$ we 
then have 
\begin{align*}
  d(x,b) 
  & =\diam\,[x,b]\ge d(x,r)\ge d(x,u)-d(r,u)\\
  &\ge \sigma S(u)-\sigma\Delta/2 = \sigma \Delta /2  \ge 
    \sigma \min\{H_\qT(b), \Delta\}/2\\
  &\ge \ga  \min\{H_\qT(b), \Delta\}. 
\end{align*}  
This is the desired inequality in this case.  

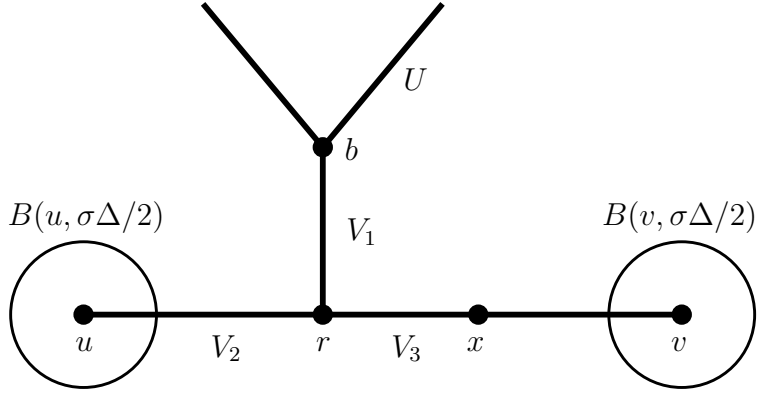
\begin{figure}
  \begin{overpic}[width=10cm, tics=10, 
    ]{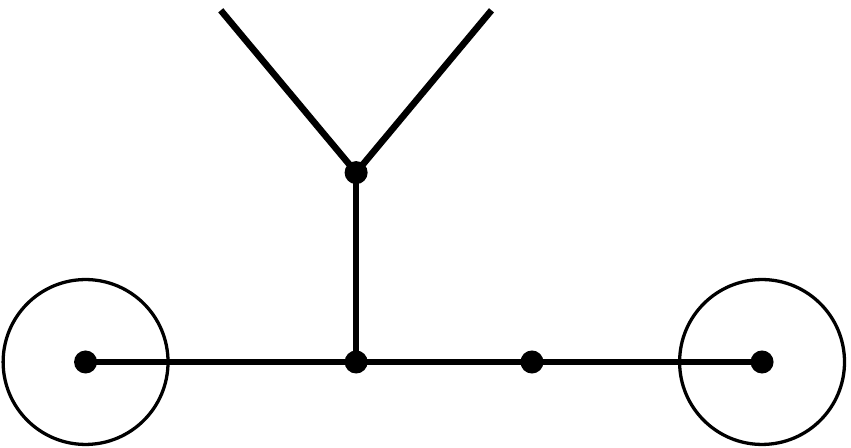}
    \put(41,5.2){$r$}
    \put(61,5.2){$x$}
    \put(0,22){$B(u,\sigma \Delta/2)$}
    \put(79,22){$B(v,\sigma \Delta/2)$}
    \put(9,5.2){$u$}
    \put(88.2,5.2){$v$}
    \put(44.8,30.7){$b$}
    \put(52.6,40){$U$}
    \put(45,20){$V_1$}
    \put(27,4.5){$V_2$}
    \put(51,4.5){$V_3$}
  \end{overpic}
  \caption{The estimate of Case 2.}
  \label{fig:prop31_case2}
\end{figure}


\smallskip
{\em Case 2.} $r\not \in  B(u, \sigma \Delta/ 2) \cup 
B(v, \sigma \Delta/ 2)$.  Then in particular $r\in \inte(J)$. There exists  a component $U$ of $\qT\setminus \{b\}$ that is disjoint from $J$ and satisfies $\diam(U) \ge H_\qT(b)$. This connected set does not contain $r\in J$ and so it is contained in a component $V_1$ of $\qT\setminus \{r\}$. Hence 
$$\diam(V_1)\ge \diam(U)\ge H_\qT(b). $$
Two other components $V_2$ and $V_3$ of   
$\qT\setminus \{r\}$ contain the half-open (and non-empty) arcs 
$[u,r)$ and $[v,r)$, respectively. The situation is illustrated
in Figure~\ref{fig:prop31_case2}.
It follows that 
\begin{equation*}
  \diam(V_2)\ge \diam[u,r)\ge d(u,r)\ge \sigma \Delta/2. 
\end{equation*}
Here we have used that $r\not \in B(u, \sigma \Delta/ 2)$. 
Similarly, 
$\diam(V_3)\ge \sigma \Delta/2$, and so 
\begin{align*} 
 H_\qT(r) & \ge \min\{ \diam(V_i): i\in \{1,2,3\}\}\\   &\ge \min \{ H_\qT(b), \sigma 
\Delta/ 2\} \ge \frac \sigma2 \min\{H_\qT(b), \Delta\}. 
\end{align*} 
It follows that 
\begin{align*} 
  d(x,b) 
  &=\diam\,[x,b]\ge d(x,r) 
    \ge  \sigma S(r)
    = \sigma \min\{H_\qT(r), \Delta\}\\
  &\ge \sigma^2 \min\{H_\qT(b),  \Delta\}/ 2
    =\gamma \min\{H_\qT(b), \Delta\},
\end{align*}
as desired. 

Note that $x\in J$ is a double point of $\qT$. Indeed, $x$ is not a branch point of $\qT$, because $x$ has a positive distance to each of them. On the other hand, 
$d(x,u)\ge \sigma S(u)=\sigma\Delta>0$, and so $x\ne u$. Similarly, $x\ne v$. Since $x\in [u,v]=J$, the points $u$ and $v$ lie in different components of  $\qT\setminus\{x\}$ by  Lemma~\ref{lem:top_T}~\ref{item:top_T1}. In particular, there are at least two, but not more than two such components. Hence $x$ is a double point of $\qT$. 
\end{proof} 

 With some small changes in the previous proof one can
  show that the set of double points $x\in
  J$ that satisfy  the estimate in 
  Proposition~\ref{prop:branchaway} is not only non-empty, but in a suitable sense actually  fairly large  (namely, {\em uniformly perfect}). Such a statement was  proved in
  recent work by Lin and Rohde (see \cite[Lemma~4.5]{LR19}).

\section{Good double points}
\label{sec:goodpts}

In this section we introduce the concept  of a ``good" double point
of our given qc-tree $\qT$. Attached to this concept are  certain numerical parameters.  The goal of this section is to show that with appropriate choices of these parameters,  one can  use a maximal set $V$ of  good double points to  obtain  a decomposition of $\qT$ with some desired geometric properties 
(see Proposition~\ref{prop:maxdouble}).

%

We  fix a scale $0<\Delta\leq \diam(\qT)=1$. We consider double
points $x\in \qT$ with the property that both components of
$\qT\setminus\{x\}$ are large, meaning that
\begin{equation}
  \label{eq:double_Delta}
  D_{\qT}(x) \geq \beta\Delta 
\end{equation}
for some constant $\beta\geq 1$ ($D_{\qT}$ was defined in \eqref{eq:DT}). We will choose $\beta$  according  to the following statement.  

\begin{proposition}
  \label{prop:V_Delta_sep}
  There is a constant  $\beta=\beta(N) \geq 1$ only depending on the doubling constant $N$ of $\qT$ such that the following statement is true: if $V\subset \qT$ is  a set of double points of $\qT$ that 
  are $\Delta$-separated and  satisfy
  \eqref{eq:double_Delta}, then
  either 
  \begin{enumerate}
  \item
    \label{item:V_Delta_sep1}
    for each component $X$ of $\qT\setminus V$
   we have 
    \begin{equation*}
      \diam(X) \le 3 \beta \Delta
    \end{equation*}
    or
  \item
    \label{item:V_Delta_sep2}
    there is an arc $I\subset \qT$ with
    \begin{equation*}
      \diam(I) \geq \Delta \text{ and } \dist(I,V) \geq \Delta,
    \end{equation*}
   and such that 
    \eqref{eq:double_Delta} holds for each double point $x\in I$ of $\qT$. 
  \end{enumerate}
 \end{proposition}

 Proposition~\ref{prop:branchaway}  implies that each arc $I\sub \qT$ contains double points of $\qT$.  So in case \ref{item:V_Delta_sep2} of Proposition~\ref{prop:V_Delta_sep}, we can add  a double point $x\in I$ of $\qT$ to $V$. 
Then this new set $V'=V\cup\{x\}$ is again a set of double points of $\qT$ that are $\Delta$-separated and satisfy
\eqref{eq:double_Delta}. This implies that for 
a maximal set $V$ as in the proposition, statement 
\ref{item:V_Delta_sep1} will always be true. 

\begin{proof}
  [Proof of Proposition~\ref{prop:V_Delta_sep}] By the doubling
  property, there exists a constant $N'=N'(N)\in \N$ only
  depending on the doubling constant $N$ of $\qT$ with the
  following property:
  if $\rho>0$ and $B\sub \qT$ is a ball in $\qT$ of radius
  $6\rho$, then every $\rho$-separated subset of $B$ contains at
  most $N'$ points. 
  We will show that the proposition
  is true with the constant $\beta = 6 N'$, which only depends on
  $N$.
 
 Let $V\subset \qT$ be a set as in the
  statement. Note that  $V$ is a finite set, because $V$ is $\Delta$-separated and $\qT$ is compact.  If all components $X$ of $\qT\setminus V$ satisfy 
  \ref{item:V_Delta_sep1}, we are done. Otherwise, 
  there exists  a component $X$ of
  $\qT\setminus V$ with $\diam(X) > 3\beta \Delta$. Then we can 
  find points $z,w\in X$ with $d(z,w)\ge 3\beta \Delta$. By 
  Lemma~\ref{lem:top_T}~\ref{item:top_T1} we then have 
  $[z,w]\cap V=\emptyset$ which implies that $J\coloneqq [z,w] \sub X$.
  
Note that  $\diam(J)\ge d(z,w)\ge  3\beta \Delta$. By decomposing $J$ into three non-overlapping subarcs of equal diameter $\ge \beta \Delta$ and trimming the ``middle" arc to appropriate size, we can find an arc 
   $J'\subset J\sub X$ with
  $\diam(J') = \beta \Delta$ that has distance $\geq \beta\Delta$
  from each of the two endpoints of $J$ (see the proof of 
  Lemma~\ref{lem:arc-decomp} for a very similar argument). This implies  that for every double
  point $x\in J'$ of $\qT$ the estimate \eqref{eq:double_Delta} holds.

 We  want to find a subarc
  $I\subset J'\sub \qT\setminus V$ with $\diam(I)\ge \Delta$ and $\dist(I,V) \geq \Delta$. To this
  end, we  fix a point $a\in J'$ as ``base point".  Now suppose  $v\in V$ is a point with $\dist(v,J') < \Delta$. If we travel from $v$ towards $a$ along  $[v,a]$, there is 
 a  first point  $r=r_v$ that belongs to $J'$ (see 
  Figure~\ref{fig:max_delta_D} for an illustration). Let
  \begin{equation*}
    R=\{r_v: v\in V,\,  \dist(v, J')< \Delta\}
  \end{equation*}
  be  the set of these ``root'' points.

  \begin{figure}
    \begin{overpic}[width=7cm,tics=10,
      ] {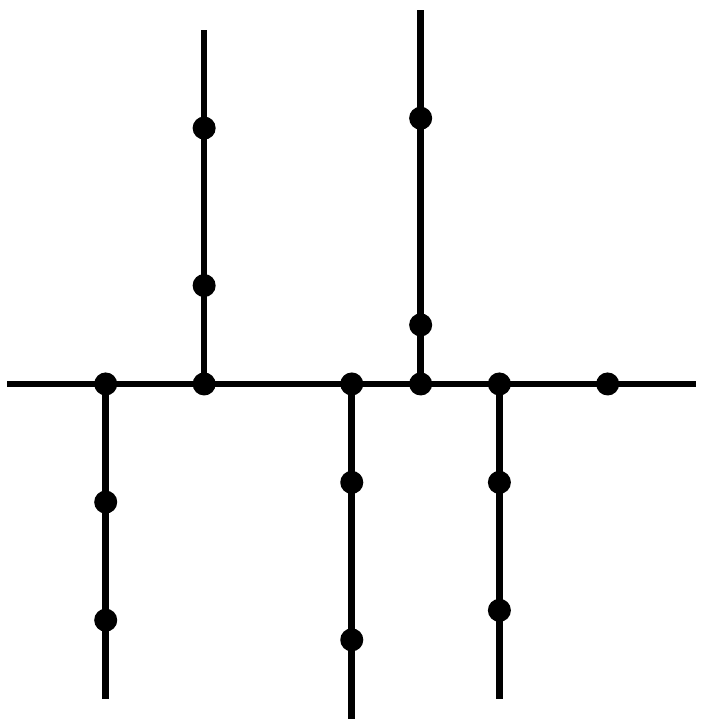}
      \put(82.1,49.7){$a$}
      \put(97,45){$J'$}
      \put(27,42){$r_v$}
      \put(30.7,59.5){$v$}
      \put(20,69){$U$}
      \put(30.7,80.7){$q_r$}
      \put(68,49.7){$r'$}
      \put(71,15){$q_{r'}$}
    \end{overpic}
    \caption{Roots in $J'$.}
    \label{fig:max_delta_D}
  \end{figure}
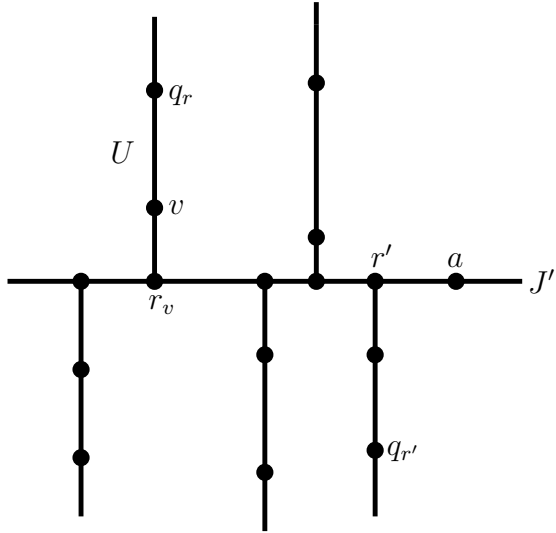


  \smallskip 
  {\em Claim.} $\#R \le N'$. 
  \smallskip

  To see this, first note that for each point $r\in R$ we can
  choose a point $v\in V$ with $d(v,r)< \Delta$ and
  $r=r_v$. The connected set $J'\sub \qT\setminus V\sub \qT\setminus \{v\}$ lies in one component of $\qT\setminus \{v\}$. Then for the other 
 component $U$ of 
  $\qT\setminus \{v\}$ we have $U\cap J' =\emptyset$ and
  $\diam(U) \ge D_\qT(v)\ge \beta \Delta$, because $v\in
  V$. Therefore,  we 
  can find a point $q\in U$ with $d(q, v)=\beta\Delta/2$ (see the proof of Proposition~\ref{prop:branchaway} for more details in a similar claim).  Define
  $v_r\coloneqq v$ and $q_r\coloneqq q$. 
  We then
  have
  \begin{align*} 
    d(q_r, a)
    &\le d(q_r, v_r)+\dist(v_r, J')+\diam(J')
    \\
    &< \beta \Delta/2+\Delta+ \beta\Delta
    \le 3 \beta \Delta.
  \end{align*}
  Thus $q_r\in B(a,3\beta\Delta)$.
  
  Moreover, if $r,r'\in R$ are distinct, then the corresponding points $q_r$ and $q_{r'}$ lie in different components of 
  $\qT\setminus \{r\}$. This can be justified by an argument similar to the one in the proof of 
  Proposition~\ref{prop:branchaway}. Hence 
    $r\in [q_r, q_{r'}]$ and so $[q_r, r]\sub [q_r, q_{r'}]$. On the other hand, $q_{r}$ and $r$ lie in different components of $\qT\setminus \{v_{r}\}$, and so 
    $v_{r} \in [q_{r}, r]\sub [q_r, q_{r'}]$. It follows that
  \begin{equation*}
    d(q_r, q_{r'})
    =
    \diam\,[q_r, q_{r'}]
    \ge
    d(q_{r}, v_r)\ge \beta \Delta/2.
  \end{equation*}
  This shows that the set $Q\coloneqq \{ q_r: r\in R\}$ consists
  of $(\beta\Delta/2)$-separated points and is contained in the
  ball $B(a, 3\beta \Delta)$. The Claim now follows from the
  definition of the constant $N'$.

  \smallskip
  By the Claim and Lemma~\ref{lem:arc-decomp} we can find an arc
  $I\sub J'$ with
  $$\diam(I) = \frac 1{6N'}\diam(J') = \frac{\beta}{6N'} \Delta =
  \Delta$$ (by choice of $\beta = 6 N'$) and
  $\dist(I, R\cup \partial J')\ge \Delta$.

  Then $\dist(I, V)\ge \Delta$. Indeed, let $v\in V$ be
  arbitrary. If $\dist(v,J')\ge \Delta$, then clearly
  $\dist(v,I)\ge \dist(v,J')\ge \Delta$.  If
  $\dist(v,J')< \Delta$, then as we travel from $v$ to a point
  in $I$ along an arc, we pass through $\partial J'$ or the root
  point $r_v\in J'\cap R$.  So
  $ \dist(v, I)\ge \dist(I, R\cup \partial J')\ge \Delta $ in this
  case as well. 
  
  Recall that $J'\subset J$ was chosen such that
  every double point of $\qT$ contained in $J'$, and hence every such
  point contained in $I\subset J'$,  satisfies
  \eqref{eq:double_Delta}. Therefore, the arc $I$ has the desired
  properties and the statement follows.
\end{proof}

In addition to \eqref{eq:double_Delta}, we want to choose double
points $x$ of $\qT$ that are separated from the branch points of $\qT$  in a controlled
way. More precisely, we require that 
\begin{equation}
  \label{eq:d_x_branch}
  d(x,b) \ge \ga   \min\{H_\qT(b),  \Delta\} 
\end{equation}
for all branch points $b\in \qT$. Here $\gamma=\ga(N)$ is the constant
from Proposition~\ref{prop:branchaway} that can be chosen to depend only on the doubling constant $N$ of $\qT$. A double point $x\in \qT$
is called \emph{$(\beta,\gamma)$-good} at scale $\Delta$, if it
satisfies \eqref{eq:double_Delta} and \eqref{eq:d_x_branch}.

\begin{proposition}
  \label{prop:maxdouble}
  Let $\beta=\beta(N)\geq 1$ be the constant from  Proposition~\ref{prop:V_Delta_sep}, $\gamma=\gamma(N)>0$ be the constant from
  Proposition~\ref{prop:branchaway}, and $0< \Delta\leq 1$. 

If  $V\sub \qT$ is a  maximal $\Delta$-separated set  of
  $(\beta,\gamma)$-good double points at scale $\Delta$, then    \begin{equation*}
    \diam (X)\le 3\beta \Delta 
  \end{equation*}
 for each component $X$ of $\qT\setminus V$. 
\end{proposition} 
Note that such a maximal set $V$ always exists, but we could very
well have $V=\emptyset$.  In this case, the statement says that
$1=\diam (\qT)\le 3\beta \Delta$. In other words, $V$ is necessarily non-empty if $0<\Delta<1/(3\beta)$. 

\begin{proof}
  Let $V$ be a set as in the statement. We argue by contradiction and assume that there is a
  component $X$ of $\qT\setminus V$ with $\diam(X) > 3\beta\Delta$. Then
  we can find  an arc $I\subset \qT$ as in
  Proposition~\ref{prop:V_Delta_sep}~\ref{item:V_Delta_sep2}.

  By Proposition~\ref{prop:branchaway} we can find a double point
  $x\in I$ of $\qT$ such that
  \begin{equation*}
    d(x,b)\ge \ga \min\{H_\qT(b), \Delta\}
  \end{equation*}
  for all branch points $b\in \qT$. Then $x$ satisfies
   \eqref{eq:double_Delta} and  \eqref{eq:d_x_branch}. Therefore, $x$ is a $(\beta, \ga)$-good
  double point of $\qT$ at scale $\Delta$. We also  have
  \begin{equation*}
    \dist (x, V)\ge \dist (I, V)\ge \Delta.
  \end{equation*}
  Hence $V'=V\cup\{x\}$ is a $\Delta$-separated set consisting of
  $(\beta, \ga)$-good double points at scale $\Delta$. Since
  $x\not\in V$, this contradicts the maximality of $V$, and the
  statement follows.
\end{proof}

At this point the importance of \eqref{eq:d_x_branch} is not at all 
obvious. The relevance of this condition will become apparent only later (see the remarks before Lemma~\ref{lem:branch_not_partial}).

\section{Subdividing the tree}
\label{sec:subdividing-tree}

We want to subdivide our given qc-tree $\qT$. As before, we
 may  assume that $\qT$ is equipped with a diameter metric $d$ and
that $\diam(\qT)=1$. We fix constants $\beta\ge 1$ and  $\ga>0$ depending
only on the doubling constant $N$ of $\qT$ as in
Proposition~\ref{prop:maxdouble}, and a (small) constant
$0<\delta<1/(3\beta)$.
 
\addtocontents{toc}{\SkipTocEntry}
\subsection*{Vertices and tiles} We will now inductively  construct  
sets $\V^n\subset \qT$ for $n\in \N$ such that
\begin{equation}
  \label{eq:Vn_inc}
   \V^1 \subset \V^2 \subset \V^3 \subset\dots,
\end{equation}
where each $\V^n$ is a maximal  $\delta^n$-separated set consisting of $(\beta,\ga)$-good double points at scale $\delta^n$.
Since $\qT$ is compact, each  set $\V^n$ will necessarily be
finite.  

For $\V^1$ we choose a maximal $\delta$-separated subset of $\qT$
consisting of $(\beta, \ga)$-good double points at scale
$\delta$.  Suppose for some $n\in \N$, the sets
$ \V^1 \subset \V^2 \subset\dots \sub\V^n \sub \qT$ with the desired
properties have been chosen. Then for
$\Delta=\delta^{n+1}\le \delta^n$ the set $\V^n$ is a
$\Delta$-separated subset of $\qT$ consisting of
$(\beta, \ga)$-good double points at scale $\Delta$. Hence it is
contained in a {\em maximal} such set. We pick such a maximal set
and denote it by $\V^{n+1}$. Clearly, $\V^n \sub \V^{n+1}$.  It
follows that we obtain sets $\V^n$ for all $n\in \N$ as
desired. Since $\delta^n\le \delta<1/(3\beta)$, we have  $\V^n\ne \emptyset $ for each 
$n\in \N$, as follows from the remark after Proposition~\ref{prop:maxdouble}.

Each  point $v\in \V^n$ is  called an \emph{$n$-vertex}.  The closure of a component of $\qT\setminus \V^n$ is  called
an \emph{$n$-tile}, and  the set of all $n$-tiles is denoted by
$\X^n$. We also speak of \emph{vertices} and \emph{tiles} if
their \emph{level} $n$ is clear from the context or irrelevant.

We now summarize some topological properties of vertices and tiles. Most of them are intuitively clear, often relying on the fact that each vertex is a double point, but we will include full proofs for the sake of completeness.  

\begin{lemma}
  \label{lem:vt} For each $n\in \N $ the following statements are true: 
  \begin{enumerate}
    \item 
    \label{item:vt1} Each $n$-tile $X$ is a subtree of $\qT$ with $ \partial X\sub \V^n$.
    
    \item
  \label{item:vt2} If $X$ is an $n$-tile and $v\in \V^n$, then
  $X$ is contained in the closure of one of the two components of
  $\qT\setminus \{v\}$ and disjoint from the other component.
  
  \item
  \label{item:vt3}
  If $X$ is an $n$-tile, then $\partial X\neq
  \emptyset$.

       \item 
    \label{item:vt4} Two distinct $n$-tiles $X$ and $Y$ have at most one point in common. Such a common point is an $n$-vertex and a boundary point of both $X$ and $Y$. 
    
      \item
    \label{item:vt5}
   Each $n$-vertex $v$ is contained in precisely two distinct $n$-tiles $X$ and $Y$.  
   
         \item
    \label{item:vt6}
 There are only  finitely many $n$-tiles.

   \item 
    \label{item:vt7} Each $(n+1)$-tile $X'$ is contained in a unique $n$-tile $X$. 
  \item
    \label{item:vt8}
    Each $n$-tile $X$ is equal to the union of all $(n+1)$-tiles $X'$
    with $X'\sub X$.  
            \item
    \label{item:vt9}
  If  $v$ is  $n$-vertex  and $X$ an $n$-tile with $v\in X$, then 
  $v\in \partial X$. Moreover, there exists precisely one
  $(n+1)$-tile $X'\subset X$ with $v\in X'$.
  
  \item
    \label{item:vt10}
     If $X$ is an $n$-tile and $\partial X=\{v\}\sub \V^n$ is a
  singleton set, then $X=\overline W$, where $W$ is a component
  of $\qT\setminus\{v\}$.

\end{enumerate}
\end{lemma} 
\begin{proof}
\ref{item:vt1} If $X$ is an $n$-tile, then $X=\overline U$, where 
$U$ is a component of $\qT\setminus \V^n$. It then follows from 
Lemma~\ref{lem:top_T}~\ref{item:top_T2} that  $X=\overline U$ is a subtree of $\qT$ with $\partial X=\partial \overline U\sub \partial U\sub \V^n$.

\smallskip \ref{item:vt2} Again we have $X=\overline U$, where
  $U$ is a component of $\qT\setminus \V^n$. Moreover,  
  $v\in \V^n$ is a double point of 
$\qT$, and so  there exist precisely two components 
$W_1$ and $W_2$ of $\qT\setminus \{v\}$. Since $U$ is a connected subset of $\qT\setminus \V^n\sub \qT \setminus \{v\}$, it is contained in one of these components, say $U\sub W_1$. 
Then $X=\overline U\sub \mybar{W}_1=W_1\cup \{v\}$, and $X$ is disjoint from $W_2=\qT\setminus   \mybar{W}_1. $



\smallskip \ref{item:vt3}
We have   $\V^n\neq \emptyset$ and so it  follows from \ref{item:vt2}
that $X\ne \qT$. Since $\qT$ is connected, this implies that $\partial X\ne \emptyset$; otherwise,  the non-empty set $X\ne \qT$ would be an open and closed  subset of the connected space $\qT$. This is impossible.

\smallskip \ref{item:vt4} This immediately follows from 
Lemma~\ref{lem:top_T}~\ref{item:top_T3}.

\smallskip \ref{item:vt5} The point $v\in \V^n$ is a double point of 
$\qT$. Hence there exist precisely two components 
$W_1$ and $W_2$ of $\qT\setminus \{v\}$. We have 
$\mybar{W}_1=W_1\cup\{v\}$ and $\mybar{W}_2=W_2\cup\{v\}$. Hence 
$v\in \mybar{W}_1\cap \mybar{W}_2$.

 Lemma~\ref{lem:top_T}~\ref{item:top_T1} implies that 
for all  points $x\in W_1$ and $y\in W_2$, we have $v\in [x,y]$. Since $v\in \mybar {W}_1\cap \mybar{W}_2$, we can choose $x$ and $y$ so close to $v$ that $[x,y]$ contains no other point  in $\V^n$. Then
$[x,v)$ is a connected subset of $\qT\setminus \V^n$, and so it must be contained in a component $U_1$ of $\qT\setminus \V^n$.
Hence  $X\coloneqq \mybar{U}_1$ is an $n$-tile that contains the arc $[x,v]$, and so $v\in X$. Similarly, $(v,y]$ is contained in a component $U_2$ of $\qT\setminus \V^n$, and $v$ is contained in the $n$-tile $Y\coloneqq \mybar {U}_2$. 

Since $v\in [x,y]$,  by Lemma~\ref{lem:top_T}~\ref{item:top_T1} the components $U_1$ and $U_2$  of $\qT\setminus \V^n$ containing $x$ and $y$, respectively, must be distinct.  So $U_1\ne U_2$, and these sets are disjoint. Since $U_1$ and $U_2$ are open by  Lemma~\ref{lem:top_T}~\ref{item:top_T2}, 
the sets $X=\mybar{U}_1 $ and $U_2$ are also disjoint. Since 
$\emptyset \ne U_2\sub \mybar{U}_2 =Y$,  we conclude that $X=\mybar{U}_1\ne \mybar{U}_2=Y$. So $v$ is contained in at least two distinct $n$-tiles $X$ and $Y$. 
 
 Suppose $Z=\overline U$ is another $n$-tile with $v\in Z$, where $U$ is a component of $\qT\setminus \V^n$. A point $z\in U$  
 must be contained in one of the components $W_1$ or $W_2$ of   
 $\qT\setminus \{v\}$, say $z\in W_1$. Then $[x,z]\sub \qT\setminus \{v\}$ by Lemma~\ref{lem:top_T}~\ref{item:top_T1}. We may assume that $x$ and $z$ are so close to $v$ that  $[x,z]$ contains no point in $\V^n\setminus\{v\}$. Then $[x,z]\cap \V^n=\emptyset$, and so $x$ and $z$ are contained in the same component of  
 $\qT\setminus \V^n$. It follows that $U=U_1$, and so 
 $X=\mybar{U}_1=\overline {U}=Z$. This shows that $X\ne Y$ are the only $n$-tiles that contain $v$. So $v$ is contained in precisely two distinct $n$-tiles.

 \smallskip \ref{item:vt6}  Each $n$-tile contains an $n$-vertex
 as follows from \ref{item:vt1} and  \ref{item:vt3}, and each
 vertex is contained in precisely two $n$-tiles by
 \ref{item:vt5}. This implies that there are at most twice as many $n$-tiles as $n$-vertices. In particular, the number of $n$-tiles is finite,
because the set $\V^n$ of $n$-vertices is finite. Actually, a more careful argument shows that the number of $n$-tiles  exceeds the number of $n$-vertices by exactly one, but we will not need this stronger result.

\smallskip \ref{item:vt7} If $X'$ is an $(n+1)$-tile, then there exists 
a component $W$ of $\qT\setminus \V^{n+1}$ with $\overline W=X'$. Since $\V^n\sub \V^{n+1}$, the set $W$ is a connected subset of $\qT\setminus \V^{n}$ and so contained in a unique component $U$ of $\qT\setminus \V^{n}$. Then 
$X'$ is contained in the $n$-tile $X\coloneqq \overline U$, because 
$ X'=\overline W\sub \overline U=X.$  
There can be no other $n$-tile containing $X'$, because by 
 \ref{item:vt1} the set $X'$ is a subtree of $\qT$ and hence an infinite set, but  distinct $n$-tiles can have at most one point in common by   \ref{item:vt4}.

\smallskip \ref{item:vt8} If $X$ is an $n$-tile, then $X=\overline U$, where $U$ is a component of $\qT\setminus \V^{n}$. Since $U$ is connected, this set cannot contain isolated points. This implies that 
the set $U\setminus \V^{n+1}\sub X$ is dense in $U$ and hence also dense in $X=\overline U$. 

If $x\in U\setminus \V^{n+1}$ is arbitrary, then there exists a component $W$ of 
$\qT\setminus \V^{n+1}$ with $x\in W$. Since  $W$ is a connected subset of $\qT\setminus \V^{n+1}\sub \qT\setminus \V^{n}$, 
this set must be contained in a component of $\qT\setminus \V^{n}$. Since $x\in U\cap W$, it follows that $W\sub U$. 
Then $X'\coloneqq \overline W$ is an $(n+1)$-tile with $x\in X'$ and $X'= \overline W\sub \overline U=X$. 

This shows that if we denote by $Y$ the union of all $(n+1)$-tiles $X'\sub X$, then $Y\sub X$  contains the set $U\setminus \V^{n+1}$. By  \ref{item:vt6} there are only finitely many $(n+1)$-tiles, and so $Y$ is closed. Since 
$U\setminus \V^{n+1}$ is dense in $X$ and $U\setminus \V^{n+1}
\sub Y$, it follows that $X=Y$ as desired.

\smallskip \ref{item:vt9} By \ref{item:vt5} there exists precisely one $n$-tile $Y$ distinct from $X$
with $v\in Y$. We then have $\{v\}=X\cap Y$ and $v\in \partial X$ by \ref{item:vt4}.
By \ref{item:vt8} there exist $(n+1)$-tiles $X'\sub X$ and $Y'\sub Y$ with $v\in X'\cap Y'$. Since $X$ and $Y$ have only the point $v$ in common, it follows that  $X'\ne Y'$ and that $Y'$ is not a subset of $X$. Since $v\in \V^n\sub \V^{n+1}$ is also an $(n+1)$-vertex, 
 \ref{item:vt5} implies that $X'$ and $Y'$ are the only $(n+1)$-tiles that contain 
 $v$. In particular, $X'$ is the unique $(n+1)$-tile with $v\in X'\sub X$.

 \smallskip \ref{item:vt10}
Suppose that $\partial X= \{v\}\sub \V^n$. We have $X=\overline U$,
where $U$ is a component of $\qT\setminus \V^n$. As we have seen in the proof of \ref{item:vt2}, there  is a component $W$ of
$\qT\setminus\{v\}$ with $U\subset W$.  We claim that $U=W$. 
 
 To see this, we argue by contradiction and assume that $U\ne W$.
 Then there exists a point $x\in U\sub W$, as well as a point 
 $y\in W\setminus U$. Hence $[x,y]\cap 
 \V^n\ne\emptyset$, because otherwise $y\in U$. 
 So as we travel from $x$ to $y$ along $[x,y]$, there must be a first point $u\in [x,y]$ that belongs to $\V^n$. Then $[x,u)\sub U$, and so $[x,u]\sub \overline U=X$. By  \ref{item:vt9}  the $n$-vertex 
 $u\in X$ is a boundary point of $X$.
 Hence $u\in \partial X=\{v\}$  and so $u=v$. 
 Since $v=u\in [x,y]$,    
  Lemma~\ref{lem:top_T}~\ref{item:top_T1} implies that $x$ and $y$ lie in different components of $\qT\setminus \{v\}$. Since $x$ and $y$ lie in the same component $W$ of $\qT\setminus \{v\}$, this is a contradiction. We see that $U=W$ and so 
  $X=\overline U=\overline W$ as desired. 
 \end{proof}

We now discuss some metric properties of vertices and tiles. 
Since 
$\V^n$ consists of $\delta^n$-separated points,  for 
distinct $u,v\in \V^n$ we have
\begin{equation}
  \label{eq:Vn_sepa}
  d(u,v) \ge \delta^n. 
\end{equation}

For  each  $n$-tile
$X^n$  we have 
\begin{align}
  \label{eq:Xn_dn}
  &\diam(X^n)\asymp  \delta^n, \text{ or more precisely }
  \\ \notag
  &\delta^n \leq \diam(X^n) \leq 3\beta \delta^n.
\end{align}
Indeed, the upper bound follows from
Proposition~\ref{prop:maxdouble}. 

To see that the lower bound is also true, first note that 
  $ \emptyset\ne \partial X^n\sub \V^n$ by
  Lemma~\ref{lem:vt}~\ref{item:vt1} and \ref{item:vt3}. 
  If  $\partial X^n$ is a singleton set $\{v\}\sub \V^n$, then $X^n$ is equal
to the closure of one of the two components of $\qT\setminus
\{v\}$ by Lemma~\ref{lem:vt}~\ref{item:vt10}. Since $v$ satisfies \eqref{eq:double_Delta}, we have
\begin{equation*}
  \diam (X^n)
  \ge
  D_{\qT}(v)
  \ge
  \beta \delta^n
  \ge
  \delta^n,  
\end{equation*}
as desired (recall that $\beta\ge 1$).   

If $\partial X^n$ contains two distinct points in $\V^n$, we
obtain the lower bound in \eqref{eq:Xn_dn} from \eqref{eq:Vn_sepa}.

We have good separation of  $n$-tiles  in the
following sense. If $X^n,Y^n\in \X^n$ are disjoint $n$-tiles,
then
\begin{equation}
  \label{eq:distXY}
  \dist(X^n,Y^n)\ge  \delta^n. 
\end{equation}
To see this, pick points $x\in X\coloneqq X^n$ and
$y\in Y\coloneqq Y^n$ such that $d(x,y)=\dist(X, Y)$. As we
travel from $x$ to $y$ along the arc $[x,y]$, we must meet the
sets $\partial X$ and $\partial Y$ because $X$ and $Y$ are
disjoint. Suppose $u\in [x,y]\cap \partial X$ and
$v\in [x,y]\cap \partial Y$. Then $u$ and $v$ are distinct
$n$-vertices and it follows that
\begin{equation*}
  \dist (X,Y)=d(x,y)=\diam\,[x,y]\ge d(u,v)\ge \delta^n,
\end{equation*} 
as desired.    
 
Since each point $v\in\V^n$ is a $(\beta, \ga)$-good double point
at scale $\delta^n$, by \eqref{eq:double_Delta}  we have 
\begin{equation}
  \label{eq:diamTv}
  D_{\qT}(v) \geq \beta \delta^n,
\end{equation}
and so the components of $\qT\setminus \{v\}$ are large.

Each  $n$-vertex $v$ stays away from the branch points of $\qT$
in a controlled way. More precisely, by  \eqref{eq:d_x_branch} for each branch point
$b$ of $\qT$ we have
\begin{equation}
  \label{eq:dist_vb}
  d(v,b)\geq \ga \min\{H_{\qT}(b), \delta^n\}.
\end{equation}

Finally, for our later discussion it is  convenient to  set $\V^0=\emptyset$ and regard $X^0\coloneqq \qT$ as the only $0$-tile. Then $\X^0=\{\qT\}$. Clearly
\eqref{eq:Xn_dn} is still true. 

\addtocontents{toc}{\SkipTocEntry}
\subsection*{Chains}
\label{sec:geometry-tiles-child}  


An \emph{$n$-chain} for $n\in \N_0$ is a finite non-empty sequence $P$ of $n$-tiles
$X_1,\dots,X_r$ with $X_i\cap X_{i+1}\neq \emptyset$ for
$i=1,\dots, r-1$. Again we call $P$ simply a \emph{chain} if its
level $n$ is clear from the context.
We call $r\in \N$ the \emph{length} of $P$.
The chain $P$ {\em joins} the points $x,y\in \qT$ if $x\in X_1$
and $y\in X_r$. It is \emph{simple} if
$X_i\ne  X_{i+1}$ for $i=1,\dots, r-1$ and 
 $X_i\cap X_j= \emptyset$ for $\abs{i-j} \geq 2$. The tiles in a simple chain $P$ are all distinct. 
 
  Given two
distinct points $x,y\in \qT$, we say that $P$ is a \emph{simple
  $n$-chain joining $x$ and $y$} if $P$ is simple, $X_1$ is
the only $n$-tile in $P$ containing $x$, and $X_r$ is the only
$n$-tile in $P$ containing $y$ (note that these requirements are stronger than saying that $P$ is simple and that $P$ joins $x$ and $y$, because the latter two conditions allow $x\in X_1\cap X_2$).

We use the notation  $\abs{P}\coloneqq\bigcup_{i=1}^r
X_i$. We say that $P$ {\em contains} a point $x$, if $x\in \abs{P}$.
Another $n$-chain $Q$ is called a {\em subchain} of $P$ if the sequence of $n$-tiles in $Q$  is obtained by deleting some of the tiles in $P$ while  keeping the order of the remaining tiles.

\begin{lemma}
  \label{lem:ex_chain_unique}
  Let $n\in \N_0$ and $x,y\in \qT$ be distinct points.  Then the
  following statements are true:
  \begin{enumerate}
  \item
    \label{item:chain1}
    There exists a unique simple $n$-chain $P$ joining $x$ and
    $y$. 
  \item
    \label{item:chain2}
    If $P$ is the simple $n$-chain and  $\widetilde{P}$ is another $n$-chain
     joining $x$ and
    $y$, then 
    $\abs{P}\subset \abs{\widetilde{P}}$. More precisely, every
    $n$-tile in $P$ also belongs to $\widetilde{P}$. 
  \end{enumerate}
\end{lemma}

We will often use the notation $P^n_{xy}$ for the unique simple
$n$-chain joining the points $x,y\in \qT$, $x\ne y$.

\begin{proof}[Proof of Lemma~\ref{lem:ex_chain_unique}]
  Let $x,y\in \qT$ with $x\ne y$ be arbitrary. We will exhibit an algorithm that produces a simple $n$-chain $P$ joining $x$ and $y$, and we will see that $P$ is the unique such $n$-chain. 
  
Let $x<v_1< \dots < v_{r-1}<y$ with $r\in \N$ be the
distinct $n$-vertices in $(x,y)$
arranged in the natural order $<$ on $[x,y]$ (obtained by identifying $[x,y]$ with the unit interval $[0,1]$). This list can  be empty (then $r=1$). We set $v_0\coloneqq x$ and $v_{r}\coloneqq y$. Then for $i=1, \dots, r$ the open arc $(v_{i-1}, v_{i})\sub [x,y]$  is a connected set in the complement of the set of $n$-vertices 
in $\qT$. Therefore, there exists a unique $n$-tile $X_i$
with $(v_{i-1}, v_{i})\sub X_i$. Then $[v_{i-1}, v_i]\sub X_i$, because $X_i$ is a closed set. For $i=1, \dots, r-1$ the $n$-vertex $v_i$ separates the sets 
$(v_{i-1}, v_{i})$ and $(v_{i}, v_{i+1})$, and so these sets must lie in different  components of $\qT\setminus\{v_i\}$ by Lemma~\ref{lem:top_T}~\ref{item:top_T1}. Closures of such components can have at most the point $v_i$ in common as follows 
from Lemma~\ref{lem:top_T}~\ref{item:top_T3}. This implies that
$X_i\cap X_{i+1}=\{v_i\}$. If $1\le i<j\le r$ and $j-i\ge 2$,
then a similar argument using a point $p\in [x,y]$ with $v_i<p<
v_{i+1}$ shows that $X_i\cap X_j=\emptyset$.  Based on this discussion one can now easily check that   the $n$-tiles $X_1, \dots, X_r$ form  a simple $n$-chain $P$ joining $x$ and $y$. 

Now suppose $\widetilde P$ is another $n$-chain joining $x$ and $y$.  Then $\abs{\widetilde{P}}$ is a path-connected set
  containing $x$ and $y$, and thus $[x,y]\subset
  \abs{\widetilde{P}}$.  In particular, 
$(v_{i-1},v_i)\subset
\abs{\widetilde{P}}$ for $i=1, \dots, r$.
%
%
Since $X_i$ contains $(v_{i-1},v_i)$ and all other $n$-tiles are
disjoint from $(v_{i-1},v_i)$ as follows from
Lemma~\ref{lem:vt}~\ref{item:vt4}, $X_i$ must be one of the $n$-tiles in $\widetilde{P}$.
Statement \ref{item:chain2} follows. 
   
Finally, to show uniqueness of $P$, assume that
$\widetilde{P}$
is a simple $n$-chain joining $x$ and $y$. Since $x=v_0\in X_1$ and $X_1$ belongs to $\widetilde{P}$, the $n$-tile $X_1$ must be the first  tile in $\widetilde{P}$. Since $X_2\ne X_1$ (in case $r\ge 2$) belongs to 
$\widetilde{P}$ and $X_1\cap X_2=\{v_1\}\ne \emptyset$, the $n$-tile $X_2$ must be the second tile in $\widetilde{P}$. Continuing in this manner, we see that the tiles in $\widetilde{P}$ are given by $X_1, \dots, X_r$. So  $\widetilde{P}=P$ and the uniqueness of $P$ follows. 
 \end{proof}

We can  construct simple $(n+1)$-chains from simple $n$-chains. 

\begin{lemma}
  \label{lem:subdiv_nchain}
  Let $n\in \N_0$ and  $x,y\in \qT$ be distinct points. Suppose the simple $n$-chain $P$ joining $x$ and $y$ consists of the $n$-tiles  $X_1,\dots,X_r$, where $r\in \N$. For $i=1,\dots,r-1$ let $v_i$ be the unique
  $n$-vertex in $X_{i}\cap X_{i+1}$, and let
  $v_0=x$ and $v_r=y$. If for $i=1, \dots , r$ we denote by $P'_i$ the simple
  $(n+1)$-chain joining $v_{i-1}$ and $v_i$, then the following
  statements 
  are true:
  \begin{enumerate}
     \item 
    \label{item:subdiv_nchain1}
    The simple $(n+1)$-chain $P'$ joining $x$ and $y$ is obtained by 
    concatenating  $P'_1,\dots,P'_r$.  
  \item 
    \label{item:subdiv_nchain2}
 The chain $P'_{i}$ consists precisely of all $(n+1)$-tiles
    $X'\sub X_i$ such that  $X'\cap (v_{i-1},v_i)\ne \emptyset$.  
    \end{enumerate}
\end{lemma}

\begin{proof}
 As in the proof of
  Lemma~\ref{lem:ex_chain_unique}, let   $v_1< \dots< 
v_{r-1}$ be  the distinct $n$-vertices in $(x,y)$ arranged in the natural order $<$ on $[x,y]$.  Then  $X_i$ is the unique $n$-tile that contains $[v_{i-1}, v_i]$ for $i=1, \dots, r$, and we have 
$\{v_i\}=X_i \cap X_{i+1}$ for $i=1, \dots, r-1$.     

We can find  the simple $(n+1)$-chain $P'$ joining $x$ and $y$ by  arranging the $(n+1)$-vertices in $(x,y)$ in the order $<$. Since the $n$-vertices $v_1, \dots, v_{r-1}$ are also $(n+1)$-vertices, they will be among  these $(n+1)$-vertices 
in $(x,y)$. 
This means that  we may assume that all the  
 $(n+1)$-vertices in $(x,y)$ are labeled $v_i^j$ so that  
\begin{align*}x&< v_1^1< \dots <v_1^{s_1-1}< v_1^{s_1}=v_1=v_2^0< \dots< v_2^{s_2}=v_2\\
&=v_3^0 <\dots< v_{r-1}^{s_{r-1}}= v_{r-1}= v_{r}^0< \dots< v_{r}^{s_r-1}<y. 
\end{align*} 
Here $s_1, \dots, s_r\in \N$. We set $v_1^0=x$ and $v_r^{s_r}=y$. 
Then the argument in the proof of  Lemma~\ref{lem:ex_chain_unique} shows that 
for $i=1, \dots, r$ and $j=1, \dots, s_i$ there exists a unique
$(n+1)$-tile $X_i^j$ such that $[v_i^{j-1}, v_i^j]\sub X_i^j$. 
Moreover, the  simple $(n+1)$-chain $P'$ joining $x$ and $y$ is given by 
$$ X_1^1, \dots, X_1^{s_1}, X_2^1, \dots, X_2^{s_2}, \dots,
X_r^1, \dots, X_r^{s_r}.$$  
It is also clear that for $i=1, \dots, r$ the simple $(n+1)$-chain 
 $P_i'$ joining $v_{i-1}=v_{i}^0$ and $v_i=v_{i}^{s_i}$ is given by 
$X_i^1, \dots, X_i^{s_i}$, because $v_i^1< \dots< v_i^{s_i-1}$ 
are all the $(n+1)$-vertices in $(v_{i-1}, v_i)=(v_i^0, v_i^{s_i})\sub (x,y)$.  This implies that $P'$ is the concatenation of the chains $P'_1,\dots,P'_r$.  Statement 
  \ref{item:subdiv_nchain1} follows.

To see \ref{item:subdiv_nchain2}, we fix $i\in \{1, \dots, r\}$.
Let  $X_i^j$ with $j\in \{1, \dots, s_i\}$ be an $(n+1)$-tile from the 
simple $(n+1)$-chain 
 $P'_i$ joining $v_{i-1}$ and $v_i$. Then 
 $X_i^j$ is contained in a unique $n$-tile by 
 Lemma~\ref{lem:vt}~\ref{item:vt7}. 
 Note that $(v_i^{j-1}, v_i^j)\sub X_i^j$ and $(v_i^{j-1}, v_i^j) \sub (v_{i-1}, v_i)$; so $X_i^j$ contains points in $(v_{i-1}, v_i)$. Since 
 $(v_{i-1}, v_i)\sub X_i$ and $(v_{i-1}, v_i)$  contains no $n$-vertices, Lemma~\ref{lem:vt}~\ref{item:vt4} implies  all  $n$-tiles except $X_i$ are disjoint from  $(v_{i-1}, v_i)$. We conclude that 
 $X_i^j\sub X_i$. This shows that  $P'_i$ consists of $(n+1)$-tiles that are contained in $X_i$ and meet  $(v_{i-1}, v_i)$. 
 
 Conversely, suppose $Z\sub X_i$ is an $(n+1)$-tile with 
 $Z\cap (v_{i-1}, v_i)\ne \emptyset$. Since 
 $$(v_{i-1}, v_i)=(v_i^0, v_i^1)\cup \{v_i^1\}\cup (v_i^1, v_i^2)\cup \dots \cup 
 \{v_i^{s_i-1}\} \cup (v_i^{s_i-1}, v_i^{s_i}),$$
 we then have $Z\cap (v_i^{j-1},v_i^j)\ne \emptyset $ for some $j\in \{1, \dots, s_i\}$ 
 or $v_i^j\in Z$ for  some $j\in \{1, \dots, s_i-1\}$.
 
 In the first case, $Z=X_i^j$, because no $(n+1)$-tile except
 $X_i^j\supset (v_i^{j-1},v_i^j) $ contains points in  $(v_i^{j-1},v_i^j)$. This follows from Lemma~\ref{lem:vt}~\ref{item:vt4}, because no point in  $(v_i^{j-1},v_i^j)$ is an
  $(n+1)$-vertex. 
 
 In the second case, when $v_i^j\in Z$ for  some $j\in \{1, \dots, s_i-1\}$, we  have $Z=X_i^j$ or $Z=X_i^{j+1}$, because 
 $X_i^j$ and  $X_i^{j+1}$ are the only $(n+1)$-tiles that contain 
 the $(n+1)$-vertex $v_i^j$ (this follows from Lemma~\ref{lem:vt}~\ref{item:vt5}).
 
 In any case, $Z$ is one of the tiles in the chain $P'_i$, and statement     \ref{item:subdiv_nchain2} follows. 
 \end{proof}

\addtocontents{toc}{\SkipTocEntry}
\subsection*{Choosing $\delta$}
\label{sec:fixing-delta}

We now are going to choose the parameter $\delta>0$ used in the definition of vertices and tiles small enough so that $(n+1)$-tiles
are contained in $n$-tiles in a ``controlled way''. The choice of $\delta$ will depend on the  
constants $\beta$ and $\gamma$ fixed at the beginning of this section.
Recall that  we imposed the 
preliminary condition  $0<\delta<1/(3\beta)$ for the definition of tiles and vertices. The ultimate  choice of $\delta$ will be discussed after the proof of Lemma~\ref{lem:branch_not_partial}.

\begin{lemma}
  \label{lem:delta_small_tiles_child}
  If $0<\delta<1/(3\beta)$ is sufficiently small only depending on the
  doubling constant $N$ of $\qT$, then the following statements are true for all   $n\in \N_0$: 
  \begin{enumerate}
  \item
    \label{item:X3child}
    Each $n$-tile $X$ contains at least three  $(n+1)$-tiles.
  \item 
    \label{item:delta_X_X1}
  If $u$ and $v$ are distinct $n$-vertices, then the simple
    $(n+1)$-chain joining $u$ and $v$ has length  $\ge 3$.
  \end{enumerate}
\end{lemma}

It follows from the first statement  that then there are least three $1$-tiles. 
The second statement implies that each 
$(n+1)$-tile $X'$ contains at most one $n$-vertex.

\begin{proof}  Fix $n\in \N_0$. Then we know by \eqref{eq:Xn_dn}
 that  $\diam(X) \geq
  \delta^n$ for each $n$-tile $X$, and $\diam(X') \leq 3 \beta \delta^{n+1}$ for each $(n+1)$-tile $X'$. 
 It follows that 
  \ref{item:X3child} is true for $0<\delta< 1/(6\beta)$. 

If  $u$ and $v$ 
are distinct $n$-vertices, then 
  $d(u,v) \ge \delta^n$ by \eqref{eq:Vn_sepa}. Again we have
  $\diam(X') \le 3\beta \delta^{n+1}$ for each
  $(n+1)$-tile $X'$. Thus \ref{item:delta_X_X1} is also  true if 
  $0<\delta< 1/(6\beta)$.
\end{proof}

In the next lemma we consider the location of $(n+1)$-vertices in
an $n$-tile. In the proof we will invoke 
\eqref{eq:dist_vb} derived from \eqref{eq:d_x_branch}.   This is the ultimate reason   
why  we want the elements in $\V^n$ to satisfy \eqref{eq:d_x_branch} in addition to \eqref{eq:double_Delta}
(with $\Delta=\delta^n$).  A consequence will be the subsequent Lemma~\ref{lem:ch+verts}. It   guarantees that if we decompose an $n$-tile $X$ into $(n+1)$-tiles, then a simple chain of $(n+1)$-tiles joining two distinct points in $\partial X$ does not encounter other points in $\partial X$.    This in turn is behind the important estimate in Lemma~\ref{lem:length_vert}~\ref{item:w_lenth2}. It   prevents blow up 
of the  auxiliary distance functions $\varrho_n$ as $n\to \infty$ that we will use to define our desired geodesic metric $\varrho$ on $\qT$ and ultimately leads to the existence of the limit in \eqref{eq:linrnexists}.


\begin{lemma}
  \label{lem:branch_not_partial}
 If $0<\delta<1/(3\beta)$ is sufficiently small only depending on the
  doubling constant $N$ of $\qT$, then the following statement is  true.  Let $n\in \N$, 
  $X$ be an $n$-tile, $u \in \partial X\sub \V^n$, and $X'\sub X$ be the unique $(n+1)$-tile with $u\in X'$. Then there exists 
  an $(n+1)$-vertex $u'\in \partial X'\setminus\{u\}$  such that $[u,u']\sub [u,v]$ for all $v\in \partial X\setminus \{u\}$.   \end{lemma}

Note that the existence of a  unique $(n+1)$-tile $X'\sub X$ with $u\in X'$ is guaranteed by 
Lemma~\ref{lem:vt}~\ref{item:vt9}. 
If we are in the setting of Lemma~\ref{lem:branch_not_partial}
and $\delta$ is so small that
 Lemma~\ref{lem:delta_small_tiles_child}~\ref{item:delta_X_X1} applies, then, as we travel from $u$ to $v$ along $[u,v]\sub X$, we must exit
$X'$, because $X'$ contains  $u$, but not $v$.  So there is a last point on $[u,v]$ that belongs to $X'$. 
This must be the point $u'\in  \partial X'\setminus\{u\}$ in the
statement, because $[u,u']\sub X'$ and $u,v$ lie in different
components of
$\qT\setminus\{u'\}$; so $X'\cap (u',v]=\emptyset$
by Lemma~\ref{lem:vt}~\ref{item:vt2}. According to Lemma~\ref{lem:branch_not_partial}, this last point $u'$ in $[u,v]\cap X'$ is independent of $v$, and so  
we always exit $X'$ at the same $(n+1)$-vertex $u'$ when traveling  from
$u$ to any other point in $\partial X$. We will later call $u$ and 
$u'$ the 
``main vertices'' of $X'$. 

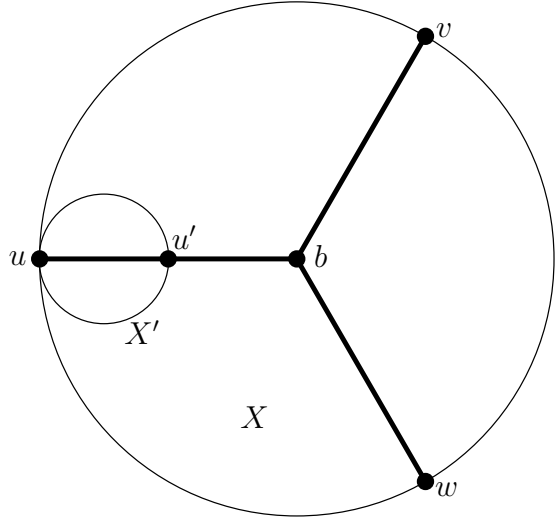
\begin{figure}
  \centering
  \begin{overpic}[width=7cm,tics=10,
    ]{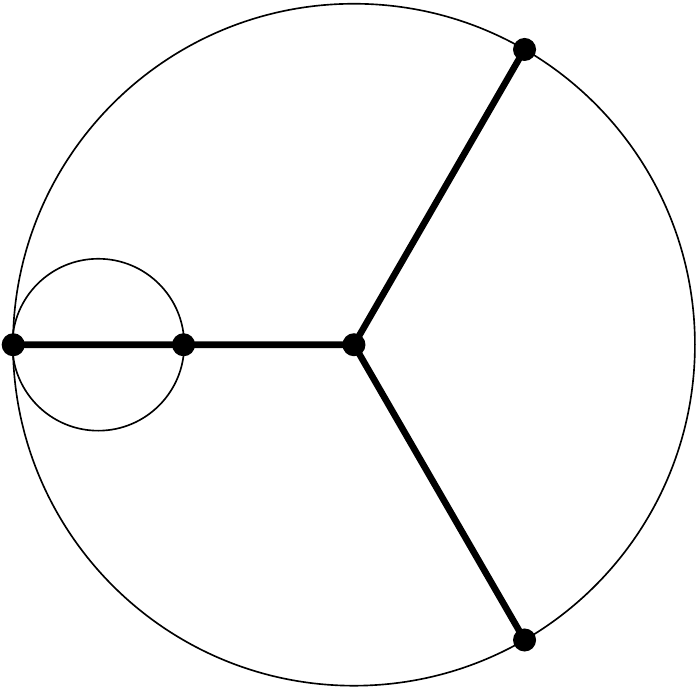}
    \put(-4,48){$u$}
    \put(27,51){$u'$}
    \put(54,47.6){$b$}
    \put(77.2,91.3){$v$}
    \put(77,4){$w$}
    \put(18,33){$X'$}
    \put(40,17){$X$}
  \end{overpic}
  \caption{Branching in an $n$-tile.}
  \label{fig:branch_ntile}
\end{figure}

\begin{proof}[Proof of Lemma~\ref{lem:branch_not_partial}] We may assume that $\delta$ is so small that the statements in Lem\-ma~\ref{lem:delta_small_tiles_child} are true. It follows from the preceding discussion, that if 
$v\in \partial X\setminus \{u\}$, then there is a last point $u'\in [u,v]\cap X'$ as we travel from $u$ to $v$ along $[u,v]$. Clearly, $u'\in \partial X'\sub \V^{n+1}$. We also have $u\ne u'$, because if $u=u'$, then 
$(u,v]\sub X$ is disjoint from $X'$
and so the set $(u,v]$ would be covered 
by the finitely many $(n+1)$-tiles $Y\sub X$ distinct from $X'$. These tiles $Y$ then also cover $[u,v]$, and so $u$ is  contained in a tile   $Y\sub X$ distinct from $X'$. We know that this is impossible, and so indeed $u'\ne u$. 

It remains to show that this last point $u'$ on $[u,v]\cap X'$ is independent of $v$. So suppose $w\ne v$ is another vertex in $\partial X\setminus \{u\}$. Then the points $u,v,w\in \partial X$ are distinct.   
For an illustration of the ensuing argument see Figure~\ref{fig:branch_ntile}.

Since $\qT$ is a tree, the arcs $[u,v]$ and $[u,w]$ share an initial 
segment $[u,b]=[u,v]\cap [u,w]$, where $b\in \qT$, but no other points.  It suffices to show that $[u,u']\sub [u,b]$. Since both points $u'$ and $b$ lie on $[u,v]$, we have $[u, u']\sub [u,b]$ or $[u,b]\sub [u,u']$. The first alternative is necessarily  true if we can show that 
$\diam\,[u,u'] < \diam\, [u,b]$. 

First note  that $b\neq u$. Indeed, if $b=u$ then $u\in [v,w]=[v,u]\cup [u,w]$, and so  $v$ and $w$ would
  lie in distinct components of $\qT\setminus \{u\}$. By Lemma~\ref{lem:vt}~\ref{item:vt2} this is impossible, because $v$ and $w$ lie in the same tile
  $X$. Similarly, $b\neq v$ and $b\neq w$.

  It follows from  Lemma~\ref{lem:tripcrit} that $b$ is a branch point of $\qT$ and $u,v,w$ lie in
  distinct components $U_u,U_v,U_w$ of $\qT\setminus\{b\}$, respectively.  Note
  that $U_u$ contains one  component $V_u$ of
  $\qT\setminus\{u\}$.  Thus by \eqref{eq:diamTv},
  \begin{equation*}
    \diam(U_u) \geq \diam(V_u) \geq D_{\qT}(u) \geq \beta
    \delta^n. 
  \end{equation*}
  Similarly, $\diam(U_v) \geq \beta \delta^n$ and $\diam(U_w)
  \geq \beta \delta^n$. It follows that
  \begin{equation*}
    H_{\qT}(b)
    \geq
    \min\{\diam(U_u),\diam(U_v),\diam(U_w)\}
    \geq
    \beta \delta^n
    \geq
    \delta^n,
  \end{equation*}
  since $\beta\geq 1$. Thus by  \eqref{eq:dist_vb} we have 
  \begin{equation*}
    d(u,b)
    \geq
    \gamma\min\{H_{\qT}(b),\delta^n\}
    \geq
    \gamma \delta^n. 
  \end{equation*}
 By  \eqref{eq:Xn_dn}  we know  that $\diam(X') \leq
  3\beta\delta^{n+1}$, and so 
  $$d(u,u')\le \diam(X')\le  3\beta\delta^{n+1}. $$ 
  So if we assume that  $0<\delta< \gamma/(3\beta)$, then it follows that  
  \begin{equation*}
 \diam\, [u,u'] =  d(u,u') \leq 3 \beta \delta^{n+1} < \gamma \delta^n
    \leq d(u,b)=\diam\, [u, b].
  \end{equation*}
  As we have seen, this implies  $[u,u']\subset [u,b]$ as desired. 
\end{proof}

For the rest of the paper we fix $0< \delta<1/(3\beta)$ such that the statements of
Lemma~\ref{lem:delta_small_tiles_child} and
Lemma~\ref{lem:branch_not_partial} are true. As we see from the proofs, it is enough to choose $\delta=\frac 12 \min\{1/(6\beta),
\gamma/(3\beta)\}$.  Then $\delta$ only depends on the doubling
constant $N$ of $\qT$, because this is true for $\beta$ and
$\gamma$. The sets ${\bf V}^n$ of vertices and 
${\bf X}^n$ of tiles for $n\in \N_0$ as constructed in the beginning of this section correspond to this choice of $\delta$ and will be fixed from now
on.

Let us record some consequences.

\begin{lemma}
  \label{lem:ch+verts} Let $n\in \N_0$, 
$X$ be an  $n$-tile, and $u,v\in \partial X\sub \V^n$ with  $u\ne
  v$. Then the simple $(n+1)$-chain $ P^{n+1}_{uv}$ joining  $u$ to $v$ consists precisely of all $(n+1)$-tiles 
  $X'\sub X$ with $X'\cap [u,v]\ne \emptyset$. 
  Moreover, $P^{n+1}_{uv}$ does not contain any point $w\in \partial X$ distinct from $u$ and $v$. 
\end{lemma}
It follows from the definition of  $P^{n+1}_{uv}$  that  only the first tile of $P^{n+1}_{uv}$ contains $u$, and only the last tile contains $v$. So $P^{n+1}_{uv}$ has ``contact" with  $\partial X$ only twice: in its first tile, where it meets $u$, and in its last tile, where it meets $v$.

\begin{proof}
  [Proof of Lemma~\ref{lem:ch+verts}]  
  Let $P\coloneqq P^{n+1}_{uv}$, and assume $P$ is given by the
  $(n+1)$-tiles $X_1,\dots,X_r$, where $r\in \N$.  \

  The first statement follows from considerations similar to the
  ones in the proof of
  Lemma~\ref{lem:subdiv_nchain}~\ref{item:subdiv_nchain2}). Note
  that Lemma~\ref{lem:vt}~\ref{item:vt9} implies that $X'=X_1$ is
  the only $(n+1)$-tile $X'\sub X$ with $u\in X'$, and $X'=X_r$
  is the only $(n+1)$-tile $X'\sub X$ with $v\in X'$.
   
  To prove the second statement, we argue by contradiction and
  assume that it is false. Then there exists a point
  $w\in \abs{P}\cap \partial X$ that is distinct from $u$ and
  $v$. Then $w\in X_i$ for some $i\in \{1,\dots,r\} $.  In fact,
  since an $(n+1)$-tile cannot contain two distinct $n$-vertices
  (see
  Lemma~\ref{lem:delta_small_tiles_child}~\ref{item:delta_X_X1}),
  we have $2\leq i \leq r-1$. Let $v_{i-1}$ and $v_i$ be the
  (unique) points in $X_{i-1} \cap X_i$ and $X_i\cap X_{i+1}$,
  respectively.
  
  We now choose the $(n+1)$-vertex $w'\in \partial X_i$ for the
  $n$-vertex $w\in X_i\sub X$ as in
  Lemma~\ref{lem:branch_not_partial}. In particular, $w'$ is the
  last point on both $[w,u]$ and $[w,v]$ as we travel from $w$ to
  $u$ or from $w$ to $v$.
  
  Since the set $X_1\cup \dots \cup X_i$ is connected, we have
  $[w,u]\subset X_1\cup \dots \cup X_i$. So the last point $w'$
  in $[w,u]\cap X_i$ must be a point in
  $X_1\cup \dots \cup X_{i-1}$.  Since $P$ is the simple
  $(n+1)$-chain joining $u$ and $v$, this is only possible if
  $w'=v_{i-1}$, because there is no other common point of $X_i$
  with any of the tiles $X_1, \dots, X_{i-1}$. Similarly, by
  considering $[w,v]\sub X_i\cup \dots \cup X_r$, we see that
  $w'=v_i$. This is impossible, because then
  $w'\in X_{i-1}\cap X_{i+1}\ne \emptyset$, contradicting the
  fact that $P$ is a simple $(n+1)$-chain.
\end{proof} 

The following statement  gives uniform control for the local combinatorics of tiles. 

\begin{lemma}
  \label{lem:nX_child_neigh}
  There is a constant $K\in \N$ such that the
  following statements are true  for each $n\in \N_0$ and each $n$-tile $X$:
  \begin{enumerate}
  \item 
    \label{item:nX_neigh}
    There are at most $K$ $n$-tiles
    that intersect $X$.
  \item 
    \label{item:nX_child}
    There are at most $K$ $(n+1)$-tiles contained in $X$.
  \end{enumerate}
\end{lemma}

\begin{proof}
   \ref{item:nX_neigh} Let  $X_1,\dots,X_k$ denote  all the  $n$-tiles distinct from $X$ that
  intersect $X$,  where  $k\in \N_0$  (if $n=0$, we have $k=0$, and  this list is empty).  Since diameters of $n$-tiles are comparable 
  to $\delta^n$ as in  \eqref{eq:Xn_dn}, there is a constant
  $C=C(N) >0$ only depending on the doubling constant $N$ of $\qT$ (and hence independent of $n$ and $X$) such that 
  \begin{equation*}
    X\cup X_1\cup \dots \cup X_k \subset B(x,C\delta^n),
  \end{equation*}
  where $x$ is some point in $X$.

 For $i=1, \dots, k$  the $n$-tiles $X$ and $X_i$ intersect in an $n$-vertex
  $v_i$ (see Lemma~\ref{lem:vt}~\ref{item:vt4}). By   Lemma~\ref{lem:vt}~\ref{item:vt5} each of these $n$-vertices $v_i$ is contained in precisely two $n$-tiles, namely $X$ and $X_i$.   It follows that $v_i\neq v_j$ for $i\neq j$. Thus
  $B(x,C\delta^n)\supset X\cup X_1 \cup \dots \cup X_k$ contains
  at least $k$ distinct $n$-vertices $v_1,\dots,v_k$.  Since 
  $C$ only depends on $N$ and 
  the $n$-vertices $v_1, \dots , v_k$ are $\delta^n$-separated by
  \eqref{eq:Vn_sepa}, it follows that there is a constant
  $K_1=K_1(N) \in \N$ 
  such that $k\leq K_1$. 

  \smallskip
  \ref{item:nX_child}
  As before, there is a constant $C=C(N)>0$ independent of $n$ and $X$
  such that $X\subset B(x,C\delta^n)$, where $x$ is some point in
  $X$ (see \eqref{eq:Xn_dn}). If $k\in \N$ is the number of
  $(n+1)$-tiles contained in $X$, then $X$ also contains at least
  $k/2$ distinct $(n+1)$-vertices, because each $(n+1)$-tile contains at least one  $(n+1)$-vertex and each $(n+1)$-vertex   is
  contained in at most two $(n+1)$-tiles.  These $(n+1)$-vertices are
  $\delta^{n+1}$-separated by \eqref{eq:Vn_sepa}.  So it follows
  that there is a constant $K_2\in\N$
  only depending on $C$ and $\delta$ (and hence independent of
  $n$ and $X$)  such that $k\leq K_2$.

  \smallskip If we now choose $K\coloneqq\max\{K_1,K_2\}$, then
  statements \ref{item:nX_neigh} and \ref{item:nX_child} are both
  true for all $n\in \N_0$ and all $n$-tiles $X$.
\end{proof}

\section{Weights and main vertices of tiles}
\label{sec:weights-vertices}

We will now define \emph{weights} of tiles. Later they will  be used to construct our desired geodesic metric $\varrho$.
The weight of each $n$-tile $X$, $n\in \N_0$, is a number
$w(X)\in (0,\infty)$. We will define it by an inductive process over the level $n\in \N_0$.   

Once we have determined weights of tiles, we can define the
\emph{$w$-length} of an $n$-chain $P$ given by the $n$-tiles
$X_1, \dots,X_r$ as
\begin{equation}
  \label{eq:deflengthP}
  \length_w(P) \coloneqq \sum_{i=1}^r w(X_i). 
\end{equation}

For the construction of the geodesic metric it is desirable to have a relation between the weight $w(X)$ of an $n$-tile $X$ and  the
$w$-length of some  simple $(n+1)$-chains $P$
joining points  on the boundary of $X$. 
For this reason, we will single out two distinct points $p,q\in 
\partial X$ (i.e., two $n$-vertices in $X$) as the
\emph{main vertices} of $X$. Of course, this requires that
$\#\partial X\geq 2$. In this case, we call  $X$ an
\emph{arc-tile}, because we think of $X$ as carrying the distinguished arc $[p,q]$.   Otherwise, $\#\partial X\ \le 1$.  If 
$\#\partial X =1$, then we call $X$ a \emph{leaf-tile}.
Finally, if  $\partial X=\emptyset$, then necessarily $n=0$ and $X=\qT$ (this follows from Lemma~\ref{lem:vt}~\ref{item:vt3}).  

If $X$ is an arc-tile, $p,q\in \partial X$ are the main vertices of 
$X$, and  $P=P^{n+1}_{pq}$ is the unique simple $(n+1)$-chain  joining $p$
and $q$, then we will choose weights in such a way that 
$\length_w(P)=w(X)$ (see
\eqref{eq:lPp_wX}). This will ensure that 
 the distance functions $\varrho_n$ that we use to define the desired geodesic metric do not  degenerate as $n\to \infty$ (see \eqref{eq:defrhon} and Lemma~\ref{rho_n_X}).

The $(n+1)$-tiles $X'\subset X$ that do not
intersect $[p,q]$ will be given a uniformly small relative weight 
$\eps_0=w(X')/w(X)$ (see \eqref{eq:weight_w_not_vw}).
 As a consequence, 
the  distance functions $\varrho_n$ are ``almost'' decreasing
(Lemma~\ref{lem:simpknchanest})  and have a limit as $n\to \infty$ (Lemma~\ref{lem:rhonlimexis}). 
Letting $\eps_0\to 0$ will later also allow us to 
derive  Theorem~\ref{thm:qtree_confdim1}.

After this outline of some of the ideas, we will now give the details for the definition of weights and main vertices of tiles. Let $K\in \N$ be the constant from
Lemma~\ref{lem:nX_child_neigh}.  We fix a parameter
\begin{equation}
  \label{eq:para}
  0<\epsilon_0\leq 1/(3K). 
\end{equation} 

There is a single $0$-tile $X^0=\qT$. We set
$w(X^0) \coloneqq 1$. Since $\partial X^0=\emptyset$, we do not
define main vertices of $X^0$.

We now assume that for some $n\in \N_0$ we have defined the
weight of each $n$-tile $X$ and its main vertices if
$\# \partial X\ge 2$. We fix $X$ and want to define weights of
$(n+1)$-tiles $X'\sub X$
and main vertices for arc-tiles $X'$.
Since every $(n+1)$-tile is contained in a unique $n$-tile (see Lemma~\ref{lem:vt}~\ref{item:vt7}), this will provide the necessary inductive step. Figure~\ref{fig:weights} illustrates how we will choose weights and main
vertices of $(n+1)$-tiles $X'\sub X$ in the ensuing discussion. 
 In this figure, we indicated relative weights $w(X')/w(X)$.

Assume first that $\partial X=\emptyset$. This happens
precisely  when $n=0$ and  $X=X^0=\qT$.   We
set $w(X')\coloneqq\eps_0 w(X)=\epsilon_0$ for each $1$-tile
$X'$. If $X'$ is an arc-tile and so $\#\partial X'\ge 2$, we pick two   (arbitrary) distinct
points in $\partial X'$ and declare them to be the main vertices of $X'$.

Suppose now that $X$ is a leaf-tile, i.e., $\#\partial X=
1$. We then set $w(X') = \epsilon_0 w(X)$ for each $(n+1)$-tile
$X'\subset X$.

To define main vertices of $(n+1)$-tiles that are  arc-tiles contained in $X$,  recall first  that there is a unique
$(n+1)$-tile $X'\subset X$ that contains the (only) $n$-vertex
$u\in \partial X$.  It follows from our choice of $\delta$ and Lemma~\ref{lem:delta_small_tiles_child}~\ref{item:X3child} that  $X'$ must be an arc-tile. 
 We
declare $u$ and some other (arbitrary) $(n+1)$-vertex
$u'\in \partial X'$ with $u'\neq u$ to be  the main vertices of $X'$.

If  an  $(n+1)$-tile  $X'\subset X$ is an arc-tile and does not intersect $\partial X$, we again declare two arbitrary
distinct $(n+1)$-vertices in $\partial X'$ to be  the main vertices
of $X'$. This completes the inductive step in the case that $X$ is a
leaf-tile.

\begin{figure}
  \centering
  \begin{overpic}
    [width=12cm, tics=10,  
    ]
    {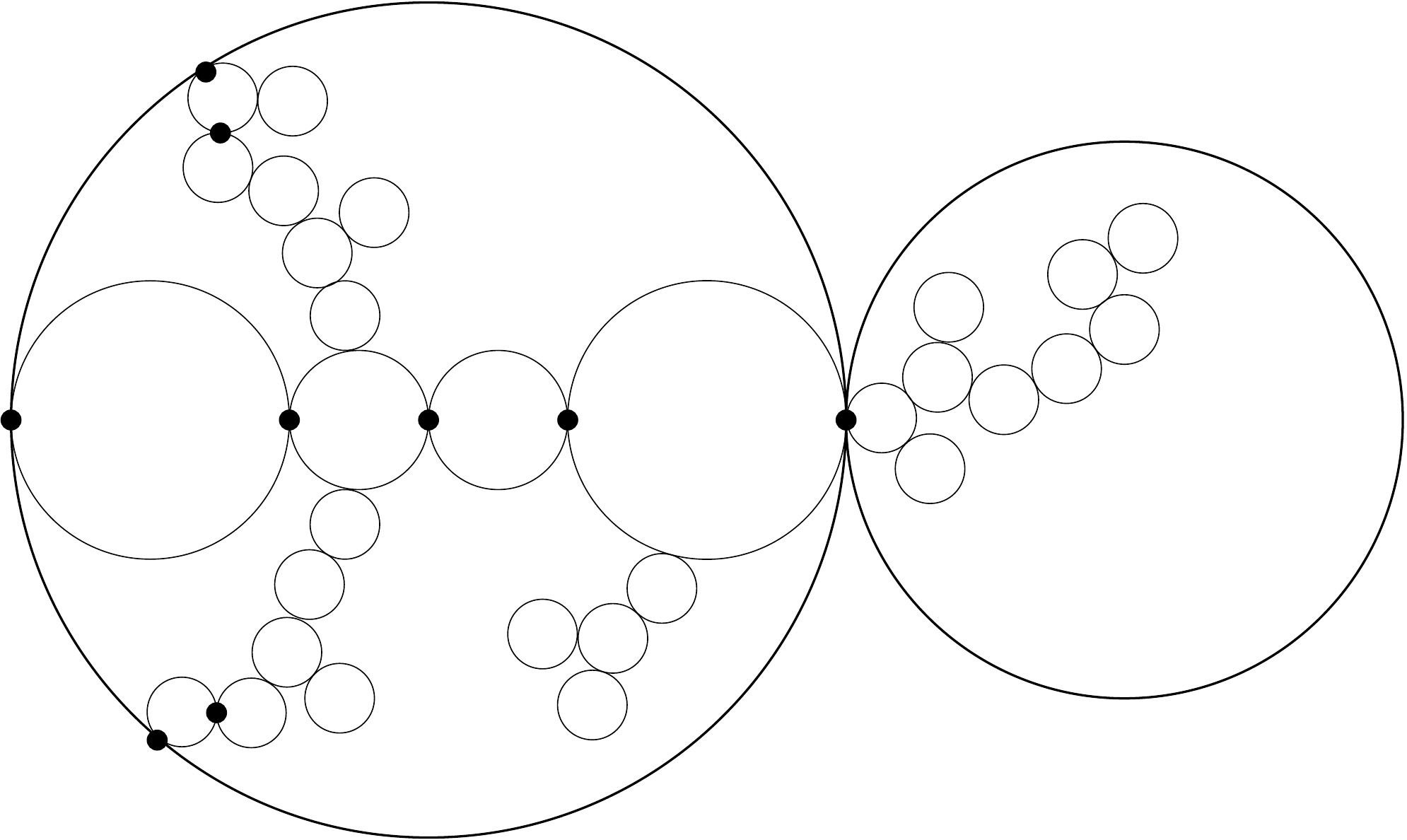} 
    \put(50, 3){$X$}
    \put(-2,29){$p$}
    \put(16.8,29){$v_1$}
    \put(31.5,29){$v_2$}
    \put(41.5,29){$v_3$} 
    \put(57.4,29){$q$}
    \put(9,29){$\tfrac{1}{3}$}
    \put(49,29){$\tfrac{1}{3}$}
    \put(22,29){${\scriptstyle\frac{1}{3(r-2)}}$}
    \put(14.6,52.1){${\scriptstyle\epsilon_0}$}
    \put(23,9.5){${\scriptstyle\epsilon_0}$}
    \put(13,56){$u$}
    \put(10.8,48.8){$u'$} 
    \put(80,5){$Y$}
    \put(61.5,29.5){${\scriptstyle\epsilon_0}$}
    \put(79,35.7){${\scriptstyle\epsilon_0}$} 
  \end{overpic}
  \caption{Relative weights and main vertices of tiles}
  \label{fig:weights}
\end{figure}
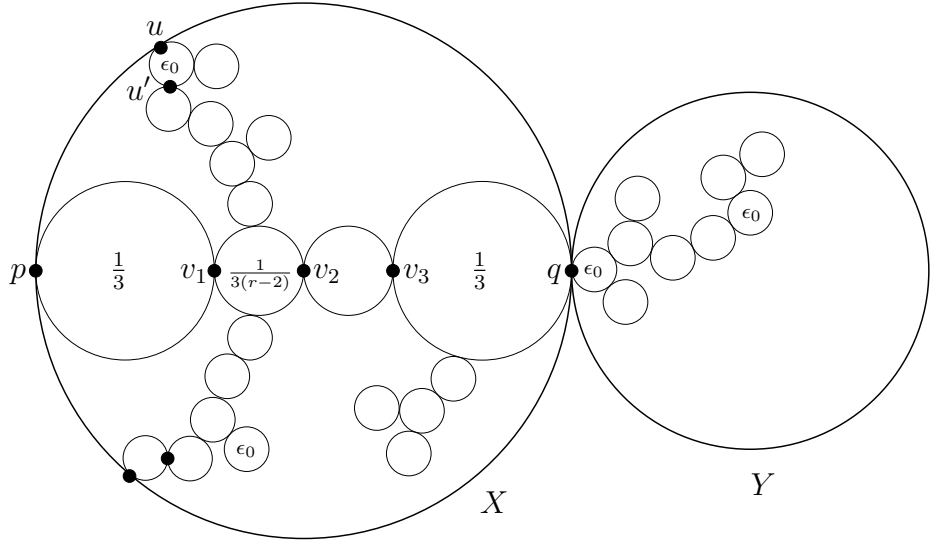

Finally,  suppose that  $X$ is an arc-tile, i.e., $\#\partial X\ge 2$, and   let 
$p,q\in\partial X$ be the main vertices of $X$.   By Lemma~\ref{lem:vt}~\ref{item:vt9} there are   unique
$(n+1)$-tiles
$X'_{p}\subset X$ and $X'_{q}\subset X$  
containing $p$ and $q$,  respectively. By our choice of $\delta$ and 
Lemma~\ref{lem:delta_small_tiles_child}~\ref{item:delta_X_X1}, the tiles $X'_p$ and $X'_q$ are distinct and disjoint. We
set
\begin{equation}
  \label{eq:weight_main_pt}
  w(X'_{p}) = w(X'_{q}) \coloneqq \tfrac{1}{3}w(X).
\end{equation}

Suppose the simple $(n+1)$-chain $P\coloneqq P^{n+1}_{pq} $ joining 
$p$ and $q$ is given by the $(n+1)$-tiles 
\begin{equation}\label{eq:P_main_vert}
  X'_1=X'_{p} ,X'_2,\dots, X'_{r}=X'_{q}.
\end{equation}
Then  $P$  consists of tiles $X_i'$ contained in $X$. 
Since $p\ne q$ we have $r\geq 3$. Moreover, $r\leq K$ by
Lemma~\ref{lem:nX_child_neigh}~\ref{item:nX_child}.
 Note that
 $P=P^{n+1}_{pq}$ consists precisely of all tiles $X'\sub X$
 with $X'\cap [p, q]\ne \emptyset$ (see
 Lemma~\ref{lem:ch+verts}).

We have $X'_1=X'_{p}$ and $X'_r=X'_{q}$ and so the weights
$w(X'_1)$ and $w(X'_r)$ are already defined. We set
\begin{equation}
  \label{eq:weight_main_diag}
  w(X'_i) \coloneqq \frac{1}{3(r-2)}w(X).
\end{equation}
for $i=2, \dots, r-1$. 
Then
\begin{equation}
  \label{eq:lPp_wX}
  \sum_{i=1}^r w(X'_i) = w(X). 
\end{equation}
So the weights are defined in such a way that the $w$-length (as  in \eqref{eq:deflengthP}) of
the simple $(n+1)$-chain $P$ joining the two main vertices $p$
and $q$ of the $n$-tile $X$ is exactly equal to $w(X)$.

To define the main vertices of the tiles $X'_i$, let $v_i$ be the
(unique) point in $X'_i\cap X'_{i+1}$ for
$i=1,\dots,r-1$. Furthermore, let $v_0\coloneqq p$ and
$v_r\coloneqq q$. Then $v_{i-1}$ and $v_{i}$ are distinct $(n+1)$-vertices  in $\partial X'_i$ for $i=1, \dots, r$. We declare them
to be the main vertices of $X'_i$. In other words, 
two successive  $(n+1)$-vertices on $[p,q]$ are the main
vertices of the unique $(n+1)$-tile $X'\subset X$ that contains them.

We now consider  an $(n+1)$-tile $X'\subset X$ that does not
intersect $[p,q]$. We set
\begin{equation}
  \label{eq:weight_w_not_vw}
  w(X') \coloneqq \epsilon_0 w(X).
\end{equation} 
It remains to define the main vertices of $X'$ if $X'$ is  an
arc-tile.  If $X'$ contains a point $u\in \partial X$, we let
$u'\in \partial X'\setminus \{u\}$  be the $(n+1)$-vertex given by
Lemma~\ref{lem:branch_not_partial}. We declare $u$ and $u'$ to be the man vertices of $X'$. If 
$X'\cap \partial X=\emptyset$, we declare two arbitrary points in
$\partial X'$ to be the main vertices of $X'$. 
This concludes the
definition of weights and main vertices of tiles $X'\sub X$ in case $X$ is an arc-tile.


The inductive step is
now complete, because we covered all possibilities for $X$. Therefore,  weights are defined for all tiles and main
vertices for all arc-tiles.  To avoid possible confusion, we point out that if an 
$n$-vertex $u$ is contained in two distinct $n$-tiles $X$ and $Y$, 
and $u$ is a  main vertex  of $X$, then it is not necessarily a main vertex  of $Y$. 
 
The choice of relative weights and main vertices is illustrated in
Figure~\ref{fig:weights}. Main vertices of $(n+1)$-tiles that intersect 
neither $[p,q]$ nor  
$\partial X$ were chosen arbitrarily,
and  are not shown in the picture. 

Note that for $r$ in \eqref{eq:weight_main_diag} we have
$3\leq r\le K$. Hence
$$\epsilon_0\leq  \frac {1}{3K} \le \frac{1}{3(r-2)} \leq \frac{1}{3}. $$ This
and the definition of the weights imply that for each
$(n+1)$-tile $X'$ contained in an $n$-tile $X$ we have
\begin{equation}  \label{eq:wXXprime}
  \epsilon_0 w(X) \leq w(X') \leq \tfrac{1}{3} w(X).
\end{equation}

Having defined the weights of tiles, we can now estimate the
$w$-length of chains (see \eqref{eq:deflengthP} for  the
definition).

\begin{lemma}
  \label{lem:length_vert}
  Let  $n\in \N_0$ and $X$ be an $n$-tile. Then the following
  statements are true:
  \begin{enumerate}
  \item
    \label{item:w_lenth1} 
    For each simple $(n+1)$-chain $P$ consisting of tiles contained  in $X$ we have
    \begin{equation*}
      \epsilon_0 w(X) \leq \length_w(P)\leq \tfrac 43 w(X).
    \end{equation*}
   
    \item
    \label{item:w_lenth2} 
    If $X$ is an arc-tile, $u,v\in \partial X$ with $u\ne v$, and
    $P=P^{n+1}_{uv}$ is the simple $(n+1)$-chain joining $u$ and
    $v$, then
    \begin{equation*}
     \length_w(P)\leq w(X).
    \end{equation*}
    Here we have equality if $u$ and $v$ are the main vertices of
    $X$.  
    \end{enumerate}
\end{lemma}

\begin{proof}
  \ref{item:w_lenth1}
  If $P$ is a simple
  $(n+1)$-chain as in the statement, then each $(n+1)$-tile
  $X'\sub X$ can appear only once in $P$.  
Recall that the number of $(n+1)$-tiles $X'\sub X$ is at most $K$; so there are at most $K$ such tiles $X'$ with 
$w(X')=\epsilon_0 w(X)$. If $X$ is an arc-tile, then among
the $(n+1)$-tiles with $X'\sub X$ there are two with 
$w(X')=\frac{1}{3} w(X)$ and an additional $r-2$ with  
$w(X')=\frac {1}{3(r-2)}w(X)$. Here $r$ is as in \eqref{eq:weight_main_diag}.
Since 
  $\eps_0K \le \tfrac 13$ by our choice of $\eps_0$, we conclude that 
  \begin{equation*}
    \eps_0 w(X)
    \le   
    \length_w(P)
    \leq 
    \left(K \epsilon_0 +\frac{2}{3} + \frac{r-2}{3(r-2)} \right)w(X) 
    \leq 
    \tfrac 43w(X),
  \end{equation*}
  and \ref{item:w_lenth1} follows.




\smallskip 
\ref{item:w_lenth2}   If $P$ is the simple $(n+1)$-chain in $X$ joining the two main
  vertices $p$ and $q$ of $X$, then we know that
  $\length_w(P)=w(X)$ (see \eqref{eq:lPp_wX}).

Suppose the simple $(n+1)$-chain $P$ joins two
  distinct $n$-vertices $u,v\in \partial X$, but not both $u$ and
  $v$ are main vertices of $X$.   Lemma~\ref{lem:ch+verts} then 
  implies that
at least one of the two  $(n+1)$-tiles $X'\sub X$ with $w(X')=\frac{1}{3}w(X)$ that contain main vertices of $X$ does not belong to $P$. We conclude that
  \begin{equation*}
    \length_w(P)\leq \left(\frac{1}{3} + \frac{r-2}{3(r-2)} + K
    \epsilon_0\right)w(X) \leq w(X). 
  \end{equation*}
  Statement  \ref{item:w_lenth2} follows. 
\end{proof}

It is in fact possible  
that for a  simple $(n+1)$-chain $P$ as in Lemma~\ref{lem:length_vert}~\ref{item:w_lenth1} we have $\length_w(P)>w(X)$. An
example can be  obtained from  Figure~\ref{fig:weights}, where one
can find  a simple $(n+1)$-chain $P$ that consists of $(n+1)$-tiles $X'\sub X$ and   contains  $P^{n+1}_{pq}$
(i.e., the simple $(n+1)$-chain  joining the main vertices $p$ and $q$ of $X$) as  a proper
subchain.

Let $X$ be an $n$-tile, and $v\in \partial X$. Then 
 $v$ is an  $n$-vertex, and so an $(n+1)$-vertex as well. 
If $X'\sub X$ is the unique $(n+1)$-tile containing $v$, then $X'$  is an arc-tile by our construction,  and  $v$ is one of the main vertices of
$X'$. Repeating this argument, we see that for each $k\ge n+1$,
$v$ is a main vertex of the  $k$-tile $X^k\subset X'$ containing
$v$, and so  \eqref{eq:weight_main_pt}  implies that 
$w(X^k) =3^{-k+n+1}w(X')$.

\begin{lemma}
  \label{lem:w_comp}
  Let $n,m\in \N_0$ with
  $\abs{n-m}\leq 1$. Suppose  $X$ is an $n$-tile, $Y$ is an $m$-tile, and  $X\cap Y\neq \emptyset$. Then
  \begin{equation*}
    w(X) \asymp w(Y),
  \end{equation*}
  where $C(\asymp)$ is independent of $n$, $m$, $X$, and $Y$.  
\end{lemma}

\begin{proof}
 We first consider the case   $n=m$. If $X=Y$ there is nothing to
  prove, and so  we assume  that $X\neq Y$. 
  There are unique tiles
  \begin{align*}
    X^0 &\supset X^1 \supset \dots \supset X^n=X
    \\
    Y^0&\supset Y^1\supset \dots \supset Y^n=Y,
  \end{align*}
  where $X^i$ and $Y^i$ are $i$-tiles for $i=0,\dots,n$. Let
  $k\leq n-1$
  be the largest number such that $X^k=Y^k$. Such a number $k$
exists, since there is only a single $0$-tile
  $X^0=Y^0=\qT$. Then $w(X^k)=w(Y^k)$. Since
  $X^{k+1},Y^{k+1}\subset X^k=Y^k$, we have $w(X^{k+1}) \asymp
  w(Y^{k+1})$ with $C(\asymp) =1/(3\epsilon_0)$,  as follows from 
  \eqref{eq:wXXprime}. 
    
  If  $k+1=n$, we are done. Otherwise, if $k+1<n$, we can again apply  \eqref{eq:wXXprime} and  obtain
  \begin{equation*}
      w(X^{k+2}) \asymp w(Y^{k+2}),
  \end{equation*}
  where $C(\asymp) = 1/(3\epsilon_0)^2$. Since $X^{k+1}\cap Y^{k+1}\supset X\cap Y\ne \emptyset$ and $X^{k+1}\ne Y^{k+1}$, there exists a unique  $(k+1)$-vertex $v$
  such that $X^{k+1}\cap Y^{k+1}=\{v\}$.  Then $X^{k+2+i}\cap Y^{k+2+i} =\{v\}$, the point $v$ is a
  main vertex 
  of $X^{k+2+i}$ and of $Y^{k+2+i}$, and so 
  \begin{equation*}
    w(X^{k+2+i}) = 3^{-i} w(X^{k+2}) 
    \text{ and } 
    w(Y^{k+2+i}) = 3^{-i} w(Y^{k+2})
  \end{equation*}
   for $i=0,\dots,n-k-2$. Thus, $w(X^n) \asymp w(Y^n)$ with  $C(\asymp)
  =1/(3\epsilon_0)^2$.

 If $|n-m|\le 1$, but $n\ne m$, we may assume that $m=n+1$. If $Y'$ is the unique $n$-tile that contains $Y$, then by the first part of the proof we have 
 $w(X)\asymp w(Y')\asymp w(Y)$ with implicit constants independent of the tiles and their levels. The statement follows. 
\end{proof}

%

\section{Construction of the geodesic metric}
\label{sec:constr-geod-metr}

Based on the concept of weights introduced in the previous
section, we can now define a new metric $\varrho$ on our given
tree $(\qT, d)$. For this purpose we first define a sequence of
distance functions $\varrho_n$ on $\qT$.

Let $n\in \N$ and $x,y\in \qT$. 
Then we define 
\begin{align}
  \label{eq:defrhon}
  \varrho_n(x,y) \coloneqq \inf\{\length_w(P):\, &
  \text{$P$ is an  } \\&\text{$n$-chain joining $x$ and $y$}\}. \notag
\end{align}
If  $x\ne y$, let 
$P^n_{xy}$ be the simple, and $P$ be an
arbitrary $n$-chain joining $x$ and $y$. Then we deduce from
Lemma~\ref{lem:ex_chain_unique}~\ref{item:chain2} that
$\length_w(P^n_{xy}) \leq \length_w(P)$. 
It follows that
\begin{equation}
  \label{eq:defrhon2}
  \varrho_n(x,y)=\length_w( P^n_{xy}) 
\end{equation}
for distinct points $x,y\in \qT$.

In this section we will show that the distance functions $\varrho_n$ have a limit $\varrho_n\ra \varrho$ as $n\to \infty$, and that $\varrho$ is a
geodesic metric on $\qT$. We will see in the next section  that
$(\qT,d)$ and  $(\qT,\varrho)$ are quasisymmetrically equivalent. Finally, in Section~\ref{sec:lower-hausd-dimens}
we will show that by choosing the parameter $\epsilon_0$ used in the definition of weights suitably small, we can
arrange  the Hausdorff dimension of $(\qT,\varrho)$ to be 
arbitrarily close to $1$. 

We start with some simple observations.

\begin{lemma}\label{lem:rhonbasics}
  For each $n\in \N$ the following statements are true: 
  \begin{enumerate}

    \smallskip 
  \item
    \label{item:rhon_nsymm}
    $\varrho_n(x,y)=\varrho_n(y,x)$ for $x,y\in \qT$.    
   
    \smallskip 
  \item 
    \label{item:rhon_ntri}
    $\varrho_n(x,y)\le  \varrho_n(x,z)+  \varrho_n(z,y)$ for $x,y,z\in \qT$. 
    
  \end{enumerate}
\end{lemma}

This shows that  $\varrho_n$ is symmetric and satisfies the triangle
inequality. However, it is not a metric. Indeed, it is immediate from the definition that 
   \begin{equation} \label{eq:defrhon1}
   \varrho_n(x,x) =\inf \{w(X) : X \text{ is an $n$-tile with $x\in X$}
   \}>0
    \end{equation}
for $x\in \qT$.

\begin{proof}[Proof of Lemma~\ref{lem:rhonbasics}]

  \ref{item:rhon_nsymm}
  Let $x,y\in\qT$ be arbitrary.
  The
  $n$-tiles $X_1, \dots, X_r$
  then 
  form an $n$-chain $P$ joining $x$
  and $y$ if and only if the $n$-tiles $X_r, \dots, X_1$ form an
  $n$-chain $\widetilde P$ joining $y$ and $x$. Moreover, we have
  $\length_w(P)= \length_w(\widetilde P)$. If we take  the
  infimum over all such $P$ here, then  \ref{item:rhon_nsymm} follows.

  \smallskip
  \ref{item:rhon_ntri}
  Let $x,y,z\in\qT$ be arbitrary. Suppose that the $n$-tiles
  $X_1, \dots, X_r$ form an $n$-chain $P$ joining $x$ and $z$,
  and the $n$-tiles $Y_1, \dots, Y_s$ form an $n$-chain
  $\widetilde P$ joining $z$ and $y$. Then the $n$-tiles
  \begin{equation*}
    X_1, \dots, X_r, Y_1, \dots, Y_s
  \end{equation*}
  form an $n$-chain $Q$ joining $x$ and $y$. Note that
  $X_r \cap Y_1\ne \emptyset$, because $z\in X_r \cap Y_1$.  We
  have $\length_w(Q)= \length_w( P)+ \length_w(\widetilde P)$.
  If we  take the infimum over all $P$ and $\widetilde P$ here,
  then \ref{item:rhon_ntri} follows.
\end{proof}

We now prepare the proof of the convergence of the sequence 
$\{\varrho_n\}$.

\begin{lemma}
  \label{lem:simpknchanest}
  Let $n,k\in \N$ with $k>n$, and $x,y\in \qT$ with $x\ne y$ be
  arbitrary. Then we have
  \begin{equation*}
    \length_w(P^k_{xy})
    \le
    \length_w(P^n_{xy})+\tfrac 12w(X)+\tfrac 12 w(Y),
  \end{equation*}
  where $X$ is the first tile in $P^n_{xy}$ and $Y$ the last tile
  in $P^n_{xy}$.
\end{lemma}

\begin{proof} 
  Let $n\in \N$ and $x,y\in \qT$ with $x\ne y$ be
  arbitrary. Suppose the simple $n$-chain $P=P^n_{xy}$
  joining $x$ and $y$ is given by the $n$-tiles $X_1,\dots,X_r$,
  where $r\in \N$. Let $X=X_1$ be the first tile and $Y=X_r$ be
  the 
  last tile in $P$. Then $x\in X$ and $y\in Y$. For
  $i=1,\dots, r-1$ let $v_i$ be the $n$-vertex where $X_i$ and
  $X_{i+1}$ intersect.  We also set $v_0\coloneqq x$ and
  $v_{r}\coloneqq y$.
  
  For $i=1, \dots, r$ let $P_i$ be the unique simple
  $(n+1)$-chain joining $v_{i-1}$ and $v_i$. Since
  $[v_{i-1},v_i]\sub X_i$, the chain $P_i$ consists of
  $(n+1)$-tiles contained in $X_i$. If we concatenate
  $P_1,\dots,P_r$, then we obtain the simple $(n+1)$-chain
  $ P^{n+1}_{xy}$ joining $x$ and $y$ (see
  Lemma~\ref{lem:subdiv_nchain}~\ref{item:subdiv_nchain1}). 

  By Lemma~\ref{lem:length_vert}~\ref{item:w_lenth2} we have
  $\length_w(P_i) \leq w(X_i)$ for $i=2, \dots, r-1$, because in
  this case the simple chain $P_i$ joins two distinct points in
  $\partial X_i$. We also have
  $\length_w(P_1)\le \tfrac 43 w(X_1)$, and
  $\length_w(P_r)\le \tfrac43 w(X_r)$ by
  Lemma~\ref{lem:length_vert}~\ref{item:w_lenth1}. It follows
  that
  \begin{align*}
    \length (P^{n+1}_{xy})
    & =
    \sum_{i=1}^r \length_w(P_i)
    \\
    &\leq
    \sum_{i=1}^rw(X_i)+\tfrac13 w(X_1)+\tfrac13 w(X_r)
    \\
    &=
    \length_w(P^n_{xy})+\tfrac13 w(X)+\tfrac13 w(Y).
  \end{align*}
  We now iterate this procedure by increasing the level by $1$ in
  each step until we reach level $k$.  In this way, we see that
  \begin{equation}
    \label{eq:knchain}
    \length_w (P^{k}_{xy})
    \le
    \length_w(P^n_{xy})+\tfrac13 \sum_{i=n}^{k-1} (w(X^i)+w(Y^i)), 
  \end{equation} 
  where $X^i$ and $Y^i$ are $i$-tiles for $i=n, \dots, k-1$ with 
  \begin{equation*}
    x\in X^{k-1}\sub \dots \sub X^{n+1}  \sub X^n=X 
  \end{equation*}
  and
  \begin{equation*}
     y\in Y^{k-1}\sub \dots  \sub Y^{n+1} \sub Y^n=Y. 
  \end{equation*}
  It follows from \eqref{eq:wXXprime} and these inclusions that
  $$ w(X^i)\le 3^{n-i} w(X) \text{ and } w(Y^i)\le 3^{n-i} w(Y) $$ for $i=n, \dots, k-1$, and so 
  \begin{align*} \sum_{i=n}^{k-1} (w(X^i)+w(Y^i)) &\le (w(X)+w(Y)) \sum_{i=n}^\infty 3^{n-i}\\ & \le \tfrac 32 (w(X)+w(Y)). 
  \end{align*}
The statement now follows from \eqref{eq:knchain}. 
\end{proof}

\begin{lemma}
  \label{lem:rhonlimexis} For all $x,y\in \qT$ the limit
 $$
    \lim_{n\to \infty} \varrho_n(x,y)\in [0, \infty)
$$
  exists. 
\end{lemma}

\begin{proof} If $n\in \N$ and  $X^n$
  is an  $n$-tile, then $w(X^n)\le 3^{-n}$,
  as follows from our definition of weights. This implies that if $x=y$, then 
  $0\le \varrho_n(x,y)\le 3^{-n}$ and so
  $\lim_{n\to \infty}\varrho_n(x,y)=0$.

  Suppose that $x\ne y$. Then it follows from \eqref{eq:defrhon2}
  and Lemma~\ref{lem:simpknchanest} that
  \begin{equation*}
    \varrho_k(x,y)\le \varrho_n(x,y)+ 3^{-n}
  \end{equation*}
  for $k,n\in \N$ with $k\ge n$.  Letting $k\to \infty$, we see
  that
  \begin{equation*}
    \limsup_{k\to \infty} \varrho_k(x,y)
    \le  
    \varrho_n(x,y)+ 3^{-n}<\infty.
  \end{equation*}
  Now letting $n\to \infty$, we conclude that
  \begin{equation*}
    \limsup_{n\to \infty} \varrho_n(x,y)
    \le
    \liminf_{n\to \infty} \varrho_n(x,y). 
  \end{equation*}
  So
  \begin{equation*}
    \limsup_{n\to \infty} \varrho_n(x,y)
    =
    \liminf_{n\to \infty} \varrho_n(x,y) <\infty. 
  \end{equation*}
  Hence
  $\displaystyle \lim_{n\to \infty} \varrho_n(x,y)$
  exists and is
  a non-negative (finite) number.
\end{proof}

We now define
\begin{equation}
  \label{eq:def_rho}
  \varrho(x,y) \coloneqq \lim_{n\to \infty} \varrho_n(x,y)
\end{equation}
for $x,y\in \qT$.
We know from Lemma~\ref{lem:rhonbasics} that $\varrho$ is a
non-negative symmetric function that satisfies the triangle
inequality.
In the proof of
Lemma~\ref{lem:rhonlimexis} we have seen that $\varrho(x,x)=0$
for $x\in \qT$.

In order to show that $\varrho$ is a metric on $\qT$, it remains
to verify that $\varrho(x,y)>0$ whenever $x,y\in \qT$, $x\ne y$.
To this end, the following estimates will be useful.

\begin{lemma}
  \label{rho_n_X} Let  $n\in \N$, and $X$ be an $n$-tile. Suppose  that $X$ is an  arc-tile.
  \begin{enumerate}
   \item
    \label{item:rho_Xn2}
    If $p$ and $q$ are the main vertices of $X$, then 
    \begin{equation*}
      \varrho_k(p,q)=w(X) 
    \end{equation*}
    for all $k\ge n$. 
 
  \item
    \label{item:rho_Xn1}
    If $u,v\in \partial X$ are two distinct $n$-vertices, then
    \begin{equation*}
      \epsilon_0 w(X) \leq   \varrho_k(u,v) \leq w(X) 
    \end{equation*}
    for all $k\ge n$. 
   \end{enumerate}
\end{lemma}

\begin{proof}
  \ref{item:rho_Xn2}
  Suppose that $p$ and $q$ are the main vertices of $X$. Then the
  simple $n$-chain $P^n$ joining $p$ and $q$ is given by the
  single tile $X$. Thus $\varrho_n(p,q) = \length_w(P^n) = w(X)$,
  and the statement is true for $k=n$.

  Suppose  the simple $(n+1)$-chain $P^{n+1}$ joining $p$ and
  $q$ is given by the $(n+1)$-tiles $X'_1,\dots,X'_r$, where
  $r\in \N$. Then
  \begin{equation*}
    \varrho_{n+1}(p,q)
    =
    \length_w(P^{n+1})
    =
    \sum_{i=1}^r w(X'_i)
    =
    w(X)
  \end{equation*}
  by \eqref{eq:lPp_wX} and \eqref{eq:defrhon2}.

  For $i=1, \dots, r-1$ let $v_i$ be the $(n+1)$-vertex where
  $X'_i$ and $X'_{i+1}$ intersect, and set $v_0\coloneqq p$, $v_r\coloneqq q$. Then 
  $v_{i-1}$ and $v_i$ are the main vertices of $X'_i$ for
  $i=1, \dots, r$ (see the discussion after \eqref{eq:lPp_wX}).

  Lemma~\ref{lem:subdiv_nchain}~\ref{item:subdiv_nchain1} implies
  that the simple $(n+2)$-chain $P^{n+2}$ joining $p$ and $q$ is
  obtained by replacing in $P^{n+1}$ the set $X'_i$ by the simple
  $(n+2)$-chain $P^{n+2}_i$ joining $v_{i-1}$ and $v_i$ for
  $i=1, \dots, r$. Since the main vertices of $X'_i$ are
  $v_{i-1}$ and $v_i$, it follows from \eqref{eq:lPp_wX} that
  $\length(P^{n+2}_i) = w(X'_i)$. This implies that
  \begin{align*}
    \varrho_{n+2}(p,q)
    = &
    \length_w(P^{n+2}) =\sum_{i=1}^r \length_w(P^{n+2}_i) 
    =
    \sum_{i=1}^r w(X'_i) 
    \\
    =&
    \length_w(P^{n+1}) = \varrho_{n+1}(p,q) 
    =
    w(X). 
  \end{align*}
  It is clear that we can repeat this argument for higher and
  higher levels, and \ref{item:rho_Xn2} follows.
 
  \smallskip
  \ref{item:rho_Xn1}
Suppose  $u,v\in \partial X$ are distinct $n$-vertices. Then the
desired upper bound follows from a 
reasoning similar to that in \ref{item:rho_Xn2} if we use the first part of
  Lemma~\ref{lem:length_vert}~\ref{item:w_lenth2} instead of
  \eqref{eq:lPp_wX} on each level.
 
In order   to verify the lower bound, we may assume $k\ge n+1$, because
  $\varrho_n(u,v)=w(X)$. Let $P^{n+1}$ be the simple
  $(n+1)$-chain joining $u$ and $v$ given by the $(n+1)$-tiles
  $Y'_1,\dots,Y'_s$, where $s\in \N$.  We know that $s\ge 3$ by
  Lemma~\ref{lem:delta_small_tiles_child}~\ref{item:delta_X_X1}
  and our choice of $\delta$.  Let $u'$ be the $(n+1)$-vertex
  where $Y'_1$ and $Y'_2$ intersect. Then $u'$ is the point given
  by Lemma~\ref{lem:branch_not_partial}. This means $u$ and $u'$
  are the main vertices of $Y'_1$ (see the discussion after
  \eqref{eq:weight_w_not_vw}). For each $k\ge n+1$ the simple
  $k$-chain $P^k_{uv}$ joining $u$ and $v$ contains the simple
  chain $P^k_{uu'}$ joining $u$ and $u'$ as a subchain, which follows from Lemma~\ref{lem:subdiv_nchain}~\ref{item:subdiv_nchain1}. Applying
  \ref{item:rho_Xn2} to the tile $Y_1'$, we see  that
  \begin{align*}
     \varrho_k(u,v)
     &=
     \length_w(P^k_{uv}) \geq \length_w(P^k_{uu'}) 
     \\
     & =\varrho(u,u')= w(Y'_1) \geq \epsilon_0 w(X). 
  \end{align*}
  This completes the proof of \ref{item:rho_Xn1}.
\end{proof}

We now introduce a quantity that will allow us to give good
estimates for distances of points in $\qT$.  For distinct
$x,y\in \qT$ we define
\begin{align}
  \label{eq:defm}
  m(x,y) 
  \coloneqq 
  \max\{n\in \N_0 : &\text{ there exist $n$-tiles $X$ and $Y$}
  \\\notag
  &\text{ with } x\in X,\, y\in Y, 
  \text{ and } X\cap Y \neq \emptyset\}. 
\end{align}
This maximum exists, because by  \eqref{eq:Xn_dn}  for $n$-tiles $X^n$ we have
\begin{equation*}
  \diam (X^n)\asymp \delta^n\to 0 \text{ as $n\to \infty$}, 
\end{equation*}
where $d$ is the underlying metric for the diameter of $X^n$.   

\begin{lemma}
  \label{lem:rho_w_m}
  Let $x,y\in \qT$ be distinct points, and $m\coloneqq m(x,y)\in \N_0$. Then
 \begin{equation*}
   d(x,y)
   \asymp
   \delta^m
   \text{ and }
   \varrho(x,y)
   \asymp
   w(X^m),
  \end{equation*}
  where $X^m$ is any $m$-tile containing $x$. Here the implicit
  constants $C(\asymp)$ are independent of $x$ and $y$.
\end{lemma}

\begin{proof}  
  By definition of $m$ there exist $m$-tiles $X$ and $Y$ with
  $x\in X$, $y\in Y$, and $X\cap Y\neq \emptyset$.
Then   by \eqref{eq:Xn_dn},
  \begin{equation*}
    d(x,y) \leq \diam_d(X) + \diam_d(Y) \asymp \delta^m. 
  \end{equation*}
  Here the implicit constant is independent of $x$ and $y$. 

  For the opposite inequality, consider $(m+1)$-tiles $X'$ and
  $Y'$ containing $x$ and $y$, respectively. Then $X'$ and $Y'$
  are disjoint by the definition of $m$, and from
  \eqref{eq:distXY} we obtain
   \begin{equation*}
     d(x,y)
     \geq
     \dist_d(X',Y')
     \ge
     \delta^{m+1}
     \asymp
     \delta^m.
  \end{equation*}
  Again the implicit constant is  independent of $x$ and
  $y$. The first statement $d(x,y)\asymp \delta^m$ follows.

  To see the statement for $\varrho$, note that one of the three chains $X$ or $Y$ or $X,Y$ is
  the simple $m$-chain joining $x$ and $y$. In any case, we
  have
  \begin{align*}
    \varrho(x,y)
    &=
      \lim_{k\to
      \infty}\varrho_k(x,y)
      \leq
      \varrho_m(x,y)  +\tfrac{1}{2}w(X)+\tfrac{1}{2}w(Y)
    \\
    & \le
      \tfrac{3}{2}(w(X) + w(Y))
      \asymp
      w(X),
  \end{align*}
  as follows from Lemma~\ref{lem:simpknchanest} and
  Lemma~\ref{lem:w_comp}. The latter lemma also implies that the
  upper bound $\varrho(x,y)\lesssim w(X)$ remains true if we
  replace $X$ with another $m$-tile containing $x$ (there are at
  most two such $m$-tiles).

  To prove the other inequality, let $X_1, \dots, X_r$ with
  $r\in \N$ be the simple $(m+1)$-chain joining $x$ and $y$. Then
  $x\in X_1$ and $y\in X_r$, and so we have
  $X_1\cap X_r=\emptyset$ by definition of $m$. Hence $r\geq
  3$. Let $u$ be the $(m+1)$-vertex in $X_1\cap X_2$ and $v$ be
  the $(m+1)$-vertex in $X_2\cap X_3$.

 Suppose that $k \geq m+1$, and consider  the simple $k$-chain $P^k_{xy}$ joining  $x$ and $y$. Then $P^k_{xy}$ contains the simple $k$-chain $P^k_{uv}$ joining $u$ and $v$ as a subchain, 
  and we see that 
  \begin{align*}
    \varrho_k(x,y) &=\length_w(P^k_{xy})\ge \length_w(P^k_{uv})
    =
    \varrho_k(u,v) \\
    &\asymp
    w(X_2)
    \asymp 
    w(X_1)
    \asymp
    w(X).
  \end{align*}
  Here Lemma~\ref{rho_n_X}~\ref{item:rho_Xn1} and
  Lemma~\ref{lem:w_comp} were used. We conclude that
  \begin{equation*}
    \varrho(x,y)
    =
    \lim_{k\to \infty} \varrho_k(x,y)
    \gtrsim  w(X).
  \end{equation*}
  In the previous inequalities, all implicit constants are
  independent of $x$ and $y$. The estimate for $\varrho$ follows.
\end{proof}

We are now ready for the main result of this section.

\begin{lemma}
  \label{lem:rho_metric}
  The distance function $\varrho$ is a geodesic metric on $\qT$.
\end{lemma}

\begin{proof}
  Lemma~\ref{lem:rho_w_m} immediately implies $\varrho(x,y) >0$ for
  distinct $x,y\in \qT$. This was the last remaining property of
  a metric we needed to verify for $\varrho$; see the discussion after
  \eqref{eq:def_rho}. Thus $\varrho$ is a metric on $\qT$.

  In order to show that $\varrho$ is a geodesic metric, we will 
  establish the following fact.

  \smallskip
  {\em Claim.} $\varrho(x,y) = \varrho(x,z) + \varrho(z,y)$,
  whenever $x,y\in \qT$, $x\ne y$, and $z\in (x,y)$. 
  
  \smallskip
  To see this, fix $n\in \N$ and let $P=P^n_{xy}$ be the simple
  $n$-chain joining $x$ and $y$ given by the $n$-tiles
  $X_1,\dots,X_r$, where $r\in \N$.  We know that these tiles cover $[x,y]$ (see the proof of Lemma~\ref{lem:ex_chain_unique}~\ref{item:chain1}), and so
  there exists a smallest number $1\le s\le r$ such that $z\in
  X_s$. Then $X_1, \dots, X_s$ is the simple $n$-chain
  $Q\coloneqq P^n_{xz}$ joining $x$ and $z$. 
  
  If $z\in X_{s+1}$ (which necessitates $r\ge s+1$), then
  $X_{s+1}, \dots , X_r$ is the simple $n$-chain
  $Q'\coloneqq P^n_{zy}$ joining $z$ and $y$. Otherwise, if
  $z\not \in X_{s+1}$, this $n$-chain $Q'$ is given by
  $X_{s}, \dots , X_r$. It now follows from \eqref{eq:defrhon2}
  that
   \begin{align*} 
     \varrho_n(x,z) + \varrho_n(z,y)
     &=
       \length_w(Q)+\length_w(Q')
     \\
     & \le
       \sum_{i=1}^sw(X_i)+    \sum_{i=s}^rw(X_i)=
       \length_w(P)+w(X_s)
     \\ 
     &\le \varrho_n(x,y)+3^{-n}. 
   \end{align*}
   Letting $n\to \infty$, we conclude that
   $\varrho(x,z) + \varrho(z,y)\le \varrho (x,y)$. Since the
   opposite inequality is true by the triangle inequality, the
   Claim follows.

   \smallskip
   Repeated application of the Claim implies that for all
   $u,v\in \qT$ the length of the arc $[u,v]$ is equal to
   $\varrho(u,v)$; in other words, $[u,v]$ is a geodesic segment
   (with respect to the metric $\varrho$) joining $u$ and
   $v$. Hence $\varrho$ is a geodesic metric.
\end{proof}

We  summarize the  properties of tiles with respect to the metric
$\varrho$. Recall that $\X^n$ denotes  the set of all $n$-tiles. 

\begin{proposition}
  \label{prop:mett} For all  $n,k\in \N_0$  the following statements are true:
  
   \begin{enumerate}
    \item 
    \label{item:mett1}   $\diam_\varrho(X)\asymp w(X)$ for all 
    $X\in \X^n$. 
     
       \item 
    \label{item:mett2}   
       $\diam_\varrho(X)\asymp \diam_\varrho(Y)$ if   $|n-k|\le 1$, $X\in \X^n$, 
       $Y\in \X^k$, 
       and $X\cap Y\ne \emptyset$.

       \item 
    \label{item:mett3}   
       $\diam_\varrho(Y)\lesssim 3^{-k} \diam_\varrho(X)$ 
         if  $X\in \X^n$, $Y\in \X^{n+k}$,  and $X\cap Y\ne \emptyset$.  

        \item
    \label{item:mett4} $\varrho(x,y)\asymp \diam_\varrho(X^m)$ for all distinct $x,y\in \qT$, where $m=m(x,y)$ and $X^m$ is 
    any $m$-tile with $x\in X^m$. 
        \end{enumerate}
        Here the implicit constants are independent of the tiles and their levels in \ref{item:mett1}--\ref{item:mett3}, and independent of 
        $x$, $y$, $X^m$  in  \ref{item:mett4}.  
\end{proposition}

\begin{proof}  \ref{item:mett1}
  We will actually show that 
  \begin{equation}
    \label{eq:Xrhisdiam0} 
    \epsilon_0^2 w(X) \leq \diam_\varrho(X)  \leq 2w(X). 
  \end{equation}

  If $x,y\in X$ with $x\ne y$, then the tile $X$ constitutes
  $P^n_{xy}$, the
  simple $n$-chain joining $x$ and $y$. So it follows from
  Lemma~\ref{lem:simpknchanest} that for $k\ge n$ we have
  \begin{equation*}
    \varrho_k(x,y)
    =
    \length_w(P^k_{xy})
    \le
    \length_w(P^n_{xy}) +w(X)
    =
    2w(X). 
  \end{equation*}
  Letting $k\to \infty$, we see that $\varrho(x,y)\le 2w(x)$, and
  so $\diam_\varrho(X) \leq 2w(X)$ as desired.

  If $X$ is an arc-tile, then it follows from
  Lemma~\ref{rho_n_X}~\ref{item:rho_Xn2} that
  $\varrho(p, q)=w(X)$ for the two main vertices $p$ and $q$
  of $X$. Hence $\diam_\varrho(X)\ge w(X)$.

  Suppose $X$ is a leaf-tile. Then $n\ge 1$ and $\partial X$ is a
  singleton set consisting of one $n$-vertex $u$. The unique
  $(n+1)$-tile $X'\sub X$ with $u\in X'$ is an arc-tile.  By
  what we have seen, it follows that
  \begin{equation*}
    \diam_\varrho(X)
    \ge
    \diam_\varrho(X')\ge w(X')
    =\epsilon_0 w(X).  
  \end{equation*}
  Finally, if $n=0$ and $X=X^0=\qT$, then $X$ contains a $1$-tile $X'$. Then $X'$ is an arc- or a leaf-tile, and  from  what we have seen, we conclude that
  \begin{equation*}
    \diam_\varrho(X)
    \ge
    \diam_\varrho(X')
    \ge
    \epsilon_0 w(X')
    =
    \eps_0^2 w(X). 
  \end{equation*}
  The statement follows. 
  
  \smallskip   \ref{item:mett2} This follows from  \ref{item:mett1}
  and Lemma~\ref{lem:w_comp}. 
  
   \smallskip   \ref{item:mett3} In the given setup, there is an $n$-tile $Y'$ with 
   $Y\sub Y'$. Then $Y'\cap X\ne \emptyset$. So   \ref{item:mett1} and  \ref{item:mett2} imply that 
   $$ \diam_\varrho (Y)\asymp w(Y)\le 3^{-k} w(Y')
    \asymp 3^{-k} \diam_\varrho (Y') \asymp  3^{-k} \diam_\varrho (X), $$ as desired.

    \smallskip   \ref{item:mett4} This follows from   
    \ref{item:mett1} and  Lemma~\ref{lem:rho_w_m}
\end{proof}

\section{Quasisymmetry}
\label{sec:quasisymmetry}

In this section we complete the proof of Theorem~\ref{main} by
showing that the original metric $d$ on $\qT$ is
quasisymmetrically equivalent to the geodesic metric $\varrho$
constructed above. For this we require the following fact.

\begin{lemma}
  \label{lem:T_rho_doubling}
  The metric space $(\qT,\varrho)$ is doubling.
\end{lemma}

\begin{proof} 
  Let $x\in \qT$ and $s>0$ be arbitrary. It suffices to show that
  the closed ball $\mybar{B}_\varrho (x,s)$ can be covered with
  a controlled number of sets of $\varrho$-diameter $<s/4$.
    
  To see this,  for $k\in \N_0$   we define 
  \begin{align*}
    U^k(x) 
    =
    \{y\in \qT: {}&\text{there exist} \text{ $k$-tiles $X$ and $Y$ }
    \\  
    &\text{with  $x\in X$, $y\in Y$, and $X\cap Y\ne \emptyset\}$}.
  \end{align*} 
  In other words,  $U^k(x)$ is the union of all $k$-tiles that meet a 
  $k$-tile containing $x$. Note that
  \begin{equation} \label{eq:Bkmxy}
     U^k(x)\setminus \{x\} = \{y\in \qT\setminus\{x\} : m(x,y) \geq k\}, 
  \end{equation}
  where  $m(x,y)$ is defined as in \eqref{eq:defm}. Indeed, if $y\in 
  \qT\setminus\{x\}$ and $m=m(x,y)\geq k$, then  
  there are non-disjoint $m$-tiles $X^m$ and $Y^m$ with
  $x\in X^m$ and $y\in Y^m$. Then the unique  $k$-tiles $X^k$ and $Y^k$ with $X^k\supset X^m$
  and $Y^k\supset Y^m$ are non-disjoint and contain $x$ and $y$
  respectively. So $y\in U^k(x)\setminus \{x\}$. Conversely, if $y\in U^k(x)\setminus \{x\}$, then $m(x,y)\ge k$ as  follows from the definitions of $U^k(x)$ and $m(x,y)$. 

  We have $U^0(x)=\qT\supset \mybar{B}_\varrho (x,s)$. Moreover, Proposition~\ref{prop:mett}~\ref{item:mett1} implies   that $\diam_\varrho (U^k(x))\to 0$ as $k\to \infty$. 
   Thus there exists a largest number
  $n\in \N_0$ with $\mybar{B}_\varrho (x,s)\sub U^n(x)$.

  By definition of $n$ we know that $\mybar{B}_\varrho (x,s) \not\subset
  U^{n+1}(x)$. This means that there is a point $y\in
 \mybar{B}_\varrho (x,s)  \setminus 
  U^{n+1}(x) \subset U^n(x)\setminus U^{n+1}(x)$. Then
  $\varrho(x,y) \leq s$,  and $m(x,y) = n$ as follows from \eqref{eq:Bkmxy}. 
  
  Let $k\in \N_0$ and $Y^{n+k}$ be an arbitrary  $(n+k)$-tile contained in an $n$-tile $Y^n\sub U^n(x)$. Then there exists an 
  $n$-tile $X^n$ with $x\in X^n$ and $X^n\cap Y^n\ne \emptyset$.
 Then it follows from Proposition~\ref{prop:mett}~\ref{item:mett2}--\ref{item:mett4}  that  
  \begin{align*}  
  \diam_\varrho(Y^{n+k})& \lesssim 3^{-k}   \diam_\varrho (Y^n) \\
  & \asymp 3^{-k}  \diam_\varrho (X^n)  \asymp 3^{-k} \varrho (x,y) \le 3^{-k} s. 
  \end{align*} 
 This estimate implies that we can find $k_0\in \N_0$ independent of $x$ and $s$ such that $\diam_\varrho(Y^{n+k_0})<s/4$ for all  
 $(n+k_0)$-tiles $Y^{n+k_0}$  contained in any $n$-tile $Y^n\sub U^n(x)$.

  The point $x$ is contained in at most two $n$-tiles, each of
  which intersects at most $K$ $n$-tiles, where $K$ is the
  constant from Lemma~\ref{lem:nX_child_neigh}. Thus $U^n(x)$ is a union
  of at most $2(K +1)$ $n$-tiles. Each of these $n$-tiles contains at most
  $K^{k_0}$ $(n+k_0)$-tiles, and all of these $(n+k_0)$-tiles have  
  $\varrho$-diameter $<s/4$. 

Hence $\mybar{B}_\varrho(x,s)\sub U^n(x)$ can be  covered by at most
  $N'\coloneqq  2(K+1)K^{k_0}$ sets of $\varrho$-diameter $<s/4$. Since  $N'$ is
  independent of $x$ and $s$, we conclude that the space
  $(\qT,\varrho)$ is doubling.
\end{proof}

\begin{lemma}
  \label{lem:id_qs}
  The identity map $\id_{\qT}\colon (\qT,d) \to (\qT,\varrho)$ is a
  quasisymmetry. 
\end{lemma}

\begin{proof} Let $x\in \qT$ and suppose that 
$\{x_n\}_{n\in \N}$ is a sequence of points with $x\ne x_n$  for $n\in \N$. Then Lemma~\ref{lem:rho_w_m} implies that, as $n\to \infty$, we have 
$d(x,x_n)\to 0$   if and only if $m(x,x_n)\to \infty$ if and only if 
$\varrho (x,x_n)\to 0$. This shows that  the metrics $d$ and $\varrho$ are topologically equivalent, and so   the map 
$\id_{\qT}\colon (\qT,d) \to (\qT,\varrho)$ is a homeomorphism.

The space $(\qT,d)$ is doubling and connected by assumption, and
$(\qT,\varrho)$ is doubling by Lemma~\ref{lem:T_rho_doubling}. So
in order to prove that
$\id_{\qT}\colon (\qT,d) \to (\qT,\varrho)$ is a quasisymmetry,
it is enough to show that it is a \emph{weak quasisymmetry} (see
\cite[Theorem~10.19]{He}). This means that we have to find a
constant $H\geq 1$ such that we have the
implication 
\begin{equation*} 
  d(x,y) 
  \leq d(x,z) 
  \Rightarrow
  \varrho(x,y) 
  \leq H 
\varrho(x,z)
\end{equation*}
for all $x,y,z\in \qT$. 

Let $x,y,z\in \qT$ with $d(x,y) \leq d(x,z)$ be arbitrary. We may
assume that the points $x,y,z$ are pairwise distinct. Let
$m\coloneqq m(x,y)$ and $n \coloneqq m(x,z)$ be defined as in 
 \eqref{eq:defm}. 
 
By the first part of Lemma~\ref{lem:rho_w_m} we have
$\delta^{m}\asymp d(x,y)\le d(x,z) \asymp \delta^{n}$. Thus there
is a constant $k_0\in\N_0$ independent of $x,y,z$ such that
\begin{equation*}
  n \le m +k_0.
\end{equation*}
 For $i=m, \dots, m+k_0$ we   choose an $i$-tile $Y^i$
 that contains $x$.  By  applying  Proposition~\ref{prop:mett}~\ref{item:mett2}
 at most $k_0$ times (and so a number of times
  independent of $x,y,z$),
 we see that 
$$ \diam_\varrho(Y^{m+k_0})\asymp  \diam_\varrho(Y^{m}).$$  
We also choose an $n$-tile $Z^n$ that contains $x$. Since $m+k_0\ge n$, and $x\in Y^{m+k_0}\cap Z^n$, 
Proposition~\ref{prop:mett}~\ref{item:mett3} implies that 
$$ \diam_\varrho(Y^{m+k_0}) \lesssim   \diam_\varrho(Z^n). $$  
So with Proposition~\ref{prop:mett}~\ref{item:mett4} we arrive at 
  \begin{equation*}
    \varrho(x,y) \asymp \diam_\varrho(Y^m) \asymp  \diam_\varrho (Y^{m+k_0}) \lesssim   \diam_\varrho(Z^n)
 \asymp \varrho(x,z).
  \end{equation*}
 Since all implicit constants in the previous 
  estimates are independent of $x,y,z$, the statement
  follows. 
\end{proof}
The proof of Theorem~\ref{main} is now complete.

\section{Lowering the Hausdorff dimension}
\label{sec:lower-hausd-dimens}

In this section we will prove Theorem~\ref{thm:qtree_confdim1}.
We assume that $\qT$ is the given qc-tree as before. Let $\alpha>1$ be arbitrary. We claim that
$\dim_H(\qT,\varrho)\le\alpha$ for the Hausdorff dimension of
$(\qT,\varrho)$ if we choose the parameter $\epsilon_0>0$ in \eqref{eq:para} that was  used in
the construction of $\varrho$ as described in the previous
sections small enough.  Then Theorem~\ref{thm:qtree_confdim1} immediately follows, because $T\coloneqq (\qT,\varrho)$ is a geodesic tree that is quasisymmetrically equivalent to $\qT$ and we have  $\dim_H T\le\alpha$.
 
As before, let $K$ be the constant from
Lemma~\ref{lem:nX_child_neigh}. Then we can
choose $\eps_0>0$ so small (in addition to our previous
requirement \eqref{eq:para})  that
 $$L \coloneqq (1/3)^{\alpha-1} 
+ K \epsilon_0^\alpha <1.$$
We claim that with these choices we have  
$\mathcal{H}^\alpha(\qT,\varrho)=0$ for the $\alpha$-Hausdorff measure of $(\qT,\varrho)$ (we will recall the relevant definitions below). This in turn implies the desired inequality $\dim_H(\qT,\varrho)\le\alpha$. 

To see that $\mathcal{H}^\alpha(\qT,\varrho)=0$, we first consider an 
 $n$-tile $X$, where $n\in \N_0$.  
 In the following estimates,  $X'$  denotes an arbitrary 
 $(n+1)$-tile with  $X'\subset
X$ and we denote by $\lambda(X')\coloneqq w(X')/w(X)$ the {\em relative weight} of $X'$. Note that $\eps_0\le \lambda (X')\le 1/3$
(see \eqref{eq:wXXprime}). 

Suppose first that $X$ is an arc-tile. Let $p$ and $q$ be the  main vertices of $X$. Then we have
$$ \sum_{X'\cap [p,q] \neq \emptyset}
   \lambda(X')=1, $$
   as follows from \eqref{eq:lPp_wX}.
   This shows that 
\begin{align*}
  \sum_{X'} w(X')^\alpha 
  &= 
  w(X)^\alpha \sum_{X'} \lambda(X')^\alpha 
  \\
  &= 
  w(X)^\alpha \bigg(\sum_{X'\cap [p,q] \neq \emptyset} \lambda(X')^\alpha
  + \sum_{X'\cap [p,q] = \emptyset} \lambda(X')^\alpha\bigg)
  \\
  &=
  w(X)^\alpha \bigg(\sum_{X'\cap [p,q] \neq \emptyset}
    \lambda(X')^{\alpha-1}\lambda(X')
  + 
  \sum_{X'\cap [p,q] = \emptyset} \epsilon_0^\alpha\bigg)
  \\
  &\leq
  w(X)^\alpha\big((1/3)^{\alpha-1} 
+ K \epsilon_0^\alpha\big)=L w(X)^\alpha.    
\end{align*}

For a leaf-tile $X$ or for the $0$-tile $X=X^0=\qT$  we have
\begin{equation*}
  \sum_{X'} w(X')^\alpha 
  \le 
  K \epsilon_0^\alpha w(X)^\alpha 
  \le 
  Lw(X)^\alpha,
\end{equation*}
and so we have the same upper bound as for an arc-tile $X$.

Now let $t>0$, and consider 
\begin{equation} 
\mathcal{H}^\alpha_t(\qT,\varrho)\coloneqq \inf\biggr\{\sum_{i\in \N} \diam_\varrho(A_i)^\alpha\biggr\}, \end{equation}
where  the infimum is taken over all countable covers 
$\{A_i\}_{i\in \N}$  of $\qT$ by sets $A_i\sub \qT$ with 
$\diam_\varrho(A_i)\le t$ for $i\in \N$.

We can choose $n\in \N$ large enough so that 
 for each $n$-tile $X$ we have 
  $$\diam_\varrho(X)\le 2w(X)\leq 2\cdot 3^{-n}\le t. $$ 
  Here we used \eqref{eq:Xrhisdiam0}  in the first inequality. Then  
  \begin{align*}
  \mathcal{H}^\alpha_t(\qT,\varrho) 
  &\leq
  \sum_{X\in \X^n} \diam_\varrho(X)^\alpha
  \le 2^\alpha 
  \sum_{X\in \X^n} w(X)^\alpha
  \leq 
 2^\alpha  L \sum_{\widetilde{X}\in \X^{n-1}}w(\widetilde{X})^\alpha
  \\
  &\leq \dots \leq  2^\alpha L^n w(X^0)^\alpha= 2^\alpha L^n,
\end{align*}
where $X^0=\qT$ is the unique $0$-tile and  $w(X^0)=1$. 
Since $L<1$, and  $n\to \infty$ as $t\to 0^+$, this implies 
\begin{equation*}
  \mathcal{H}^\alpha(\qT,\varrho) \coloneqq  \lim_{t \to 0^+}
  \mathcal{H}_t^\alpha(\qT,\varrho) =2^\alpha \lim_{n\to \infty} L^n=0, 
\end{equation*}
as desired. The proof of Theorem~\ref{thm:qtree_confdim1} is now
complete. 

\section{Remarks and open problems}
\label{sec:open-probl-concl}

The general strategy to prove the
quasisymmetric equivalence of $(\qT,d)$ and $(\qT,\varrho)$
follows a pattern that has been employed before (for example, see 
the proof of \cite[Theorem~18.1]{BM17}). In the follow-up paper
\cite{BM}  we will 
state   general conditions that ensure quasisymmetric
equivalence in similar situations. This approach is closely related to
recent work by Kigami (see \cite{Ki}). 


It is an interesting problem  whether every  qc-tree $\qT$ admits a
quasisymmetric embedding $\varphi\: \qT \ra \C$ into the complex
plane and whether one can obtain an image $T\coloneqq\varphi(\qT)$ with good geometric
properties. For example, one can  ask whether for a suitable quasisymmetric embedding $\varphi$ the image $T$ is quasi-convex with respect to the
Euclidean metric (then $T$ is geodesic if equipped with its
internal path metric) and $\C\setminus T$ is a nice domain (such
as a John domain).

For a given tree $\qT\subset \C$, we may consider the conformal map
$\varphi\colon \CDach \setminus\overline{\D} \to
\CDach\setminus\qT$. Here $\CDach=\C\cup\{\infty\}$ is the Riemann sphere and $\D=\{z\in \C: |z|<1\}$ the unit disk. Since  $\qT$ is locally
connected, $\varphi$ extends to a map on the boundary $f\colon
\partial \D\to \qT$ by Carath\'eodory's theorem and one obtains  an  equivalence
relation on $\partial \D$ given by $s\sim t \Leftrightarrow
f(s)=f(t)$.

Lin and Rohde have recently studied which equivalence relations
$\sim$ arise in this way from trees $\qT\subset \C$, where
$\CDach \setminus \qT$ is a John domain (see \cite{LR19}). In
particular, they were interested in related questions for the
\emph{continuum random tree} (CRT) (see \cite{BT18} for
references and relevant facts about the CRT in a related
setting). The CRT is geodesic, but not doubling, and so not a
qc-tree according to our terminology.

This leads to the question, whether a tree that is of bounded
turning,  but not necessarily doubling, admits a uniformization similar to the
one in Theorem~\ref{main}. In \cite{Me11} it is shown that an
arc is of bounded turning if and only if it is the
 image of $[0,1]$ under a weak quasisymmetry. In analogy,  one may ask whether a tree
that is of bounded turning is the  image of a
geodesic tree under a weak quasisymmetry.

Trees and tree-like spaces   often appear in  data structures. Our subdivision procedure as described in 
Section~\ref{sec:subdividing-tree} essentially gives 
an algorithm to decompose trees with good geometric control.
It would be interesting to see whether this procedure has applications in a data-related  context. 

\medskip 
{\bf Acknowledgments.} We  would like to thank Steffen Rohde for
some interesting discussions, in particular about
Lemma~\ref{lem:Sonne}. We are also grateful to Guy C.\ David for
reminding us of the reference \cite{Kin}, which we had overlooked
in a first version of this paper. Finally, we thank the anonymous
referee and  the editor Henryk Toru\'{n}czyk for various
helpful comments. 

M.B.\ was partially supported by NSF grants DMS-1506099 and  DMS-1808856.

\end{document}